\documentclass[10pt]{article}

\usepackage[latin2]{inputenc}
\usepackage{amsmath,amsfonts}
\usepackage{graphicx}
\usepackage{amssymb}
\usepackage{esint}
\usepackage{amsthm}
\usepackage{mathtools}%added to move equations to the left within the align environment
\usepackage{tikz}
\usepackage{epsfig}
\usepackage[english]{babel}
\usepackage{amsfonts}
\usepackage{amsmath}
\usepackage{amssymb}
\usepackage{graphicx}
\usepackage{mathrsfs}
\usepackage{dsfont}

\newtheorem{theorem}{Theorem}
\newtheorem{proposition}[theorem]{Proposition}
\newtheorem{lemma}[theorem]{Lemma}

\newtheorem{remark}[theorem]{Remark}
\newtheorem{example}[theorem]{Example}
\newtheorem*{theorem*}{Theorem}

\def\XXint#1#2#3{{\setbox0=\hbox{$#1{#2#3}{\int}$ }
\vcenter{\hbox{$#2#3$ }}\kern-.6\wd0}}

\newcommand{\sign}{\operatorname{sign}}

\newcommand{\essinf}{\operatorname{essinf}}
 
\newcommand{\R}{\mathbb{R}}
\newcommand{\Z}{\mathbb{Z}}

\newcommand{\N}{\mathbb{N}}

\def\ae{\color{black}}

\usepackage{a4wide}
\begin{document}
\title{External forces in the continuum limit of discrete systems with non-convex interaction potentials: Compactness for a $\Gamma$-development}
\author{Marcello Carioni\footnote{Marcello Carioni, University of
W\"urzburg, Institute of Mathematics, Emil-Fischer-Str.\ 40, 97074 W\"urzburg, Germany.  Present address: University of Graz, Institute for Mathematics and
Scientific Computing, 
Heinrichstra{\ss}e 36,
8010 Graz, Austria, Email: {\tt marcello.carioni@uni-graz.at} },~ Julian Fischer\footnote{Julian Fischer, Max Planck Institute for Mathematics in the Sciences, Inselstr.\ 22, 04103 Leipzig, Germany.  Present address: Julian Fischer
Institute of Science and Technology Austria, 
Am Campus 1,
3400 Klosterneuburg, Austria, Email: {\tt julian.fischer@ist.ac.at}  \ae}~
and Anja Schl\"omerkemper\footnote{Anja Schl\"omerkemper, University of
W\"urzburg, Institute of Mathematics, Emil-Fischer-Str.\ 40, 97074 W\"urzburg, Germany, Email: {\tt anja.schloemerkemper@mathematik.uni-wuerzburg.de}}}
\date{\today}
%\keywords{}
\maketitle

\begin{abstract}
\noindent This paper is concerned with equilibrium configurations of one-dimensional particle systems with non-convex nearest-neighbour and next-to-nearest-neighbour
interactions and its passage to the continuum. 
The goal is to derive compactness results for a $\Gamma$-development of the energy with the novelty that external forces are allowed.  In particular, the forces may depend on Lagrangian or Eulerian coordinates and thus may model dead as well as live loads. \\
Our result is based on a new technique for deriving compactness results
which are required for calculating the first-order $\Gamma$-limit in the presence of external forces: instead of comparing a configuration of $n$ atoms to a global minimizer of the
$\Gamma$-limit, we compare the configuration to a minimizer in
some subclass of functions which in some sense are ``close to'' the
configuration. The paper is complemented with the study of the minimizers of the $\Gamma$--limit. 
\end{abstract}

\section{Introduction}
The derivation of continuum theories from atomistic models in the context of elasticity theory has been a
very active area of research in the previous decades. Prominent mathematical methods, as well as the present paper, are phrased in the context of $\Gamma$-convergence; see, e.g., \cite{Braides} for an introduction. One important problem for which partial results are available is the
derivation of continuum models for brittle fracture as a limit of atomistic
models. In the limit we expect a variational problem that yields information on the cracks and models the elastic behaviour of the material outside the crack. %Such variational models are often referred to as Griffith energies, cf. \cite{FrancfortMarigo, Trusk1996}. 

In this article we focus on one-dimensional systems, which have served as toy-models by themselves in engineering and mathematical physics, see, e.g.\ \cite{CharlotteTrusk2002,JansenKoenigSchmidtTheil} and references therein. In particular, one-dimensional systems have the advantage of a monotonic ordering of the particles that makes the mathematical modeling simpler, cf.\ \cite{BraidesGelli2017,BraidesSolci2014}. In nanomaterials research, one-dimensional particle systems have recently gained interest, for instance as carbon atom wires \cite{ReviewCAW, Nairetal,Zhangetal} or one-dimensional systems of silicon \cite{Sylicon}.
%Linear carbon structures are believed to be observable in the universe as fragments of fullerenes \cite{Strelnikovetal}. 
One-dimensional surface nanostructures on semiconductor surfaces such as chains of gold atoms on certain substrates, see, e.g., \cite{Wagneretal}, raise the interest of mechanical properties of such chains depending on the forces exerted by the substrate since the electronic properties strongly depend on the lattice constants. Similarly, chains of fullerenes in carbon nanotubes experience external forces which influence the chain's mechanical properties, see, e.g., \cite{Warneretal}. Our analysis yields a first insight into the effective behaviour due to particle interactions and external forces.
\\

Following \cite{BraidesCicalese,
ScardiaSchloemerkemperZanini}, we focus on one-dimensional systems of particles (e.g.~atoms or molecules) which interact with their nearest and next-to-nearest neighbours via some non-convex potential like the classical Lennard-Jones potential. We additionally allow for  external forces, including dead as well as live loads. That is, the external forces may depend on the position within the reference configuration of the system as well as on the deformation, i.e., on Lagrangian and Eulerian coordinates.  

We stress that the class of interaction potentials  between the particles which  we consider in our analysis includes many physically interesting potentials that are strongly repulsive at short distances and mildly attractive at long distances, cf.\ Section~\ref{assumptions} for more details. The classical Lennard-Jones potential (even without a cut-off) is just one example. Other expamples include interaction potentials typically used in molecular dynamics like the Gay-Berne potentials for fixed geometries of the molecules in liquid crystals \cite{Zannoni2001} or the Girifalco potentials derived for the interaction of fullerene molecules \cite{Girifalco}. \\

Our aim is to derive continuum energy functionals from models in the discrete setting as the number of particles tends to infinity. We do so in the context of $\Gamma$-convergence techniques. Related mathematical results known for higher-dimensional fracture problems and the derivation of surface energies in the passage from discrete to continuous systems can be found in \cite{BraidesGelli2017, FriedrichSchmidt2014, FriedrichSchmidt2015,KitavtsevLuckhausRueland}. For literature on the variational modeling of fracture in the continuum setting we refer to the introduction of \cite{Friedrich2017} and the references therein.

In the setting of one-dimensional chains of atoms as considered here, there are essentially two approaches that lead to a limiting functional that contains information on the elastic behaviour outside cracks as well as information on the cracks. (1) One approach is starting from an energy density of the discrete system which is given by the sum of all interaction potentials between the atoms of the chain. The $\Gamma$-limit of this energy yields an integral functional that corresponds to a bulk energy contribution only; information on the size of the fracture is lost since the integrand of that bulk term ($d=1$-dimensional) is the convex hull of an effective potential which has its minimum on an unbounded set.  One therefore considers the so-called $\Gamma$-limit of first order which recovers some information on surface energies ($d-1=0$-dimensional), i.e. on the energies related to the crack formation. A $\Gamma$-development then yields the desired limiting model, cf.\  \cite{AnzellottiBaldo,BraidesTruskinovsky, SSZ12}.  We note that this first approach, which we follow here, allows for large deformations, cf.\ also \cite{BraidesCicalese,ScardiaSchloemerkemperZanini}. For a discussion of the importance of allowing for large deformations see \cite{Friedrich2017}. (2) The second approach starts from suitably rescaled energy densities which essentially scale surface and bulk contributions in the same way. The $\Gamma$-limit of such rescaled energies leads to a contribution of a linear elastic energy and a part depending on the cracks. The latter approach involves a harmonic approximation around the minimum of the interaction potentials and thus can be considered as an approach in small  displacements, cf.\  \cite{BraidesLewOrtiz, BraidesTruskinovsky, LauerbachSchaffnerSchlomerkemper,SSZ12,SchaeffnerSchloemerkemper}. Moreover, there are studies of one-dimensional systems with different scalings for the convex and the concave part of the internal potentials, cf.\  \cite{BraidesDalmasoGarroni,BraidesGelli2002, BraidesGelli}. For other mathematical approaches to discrete to continuum analysis for fracture mechanics and surface energies we refer to \cite{CharlotteTrusk2002, CharlotteTrusk2008, Hudson,IosifescuLichtMichaille,JansenKoenigSchmidtTheil}.   \\
 %Atomic chains with interactions between $K>2$ neighbours were treated in \cite{SchaeffnerSchloemerkemper}. In computational mechanics one often considers hybrid models; these were mimicked in a discrete-to-continuum limit in \cite{SchSch2015}. 
 %Heterogeneous materials and their continuum limit as well as homogenization in the context of fracture were studied in \cite{LauerbachSchaffnerSchlomerkemper}. With the current paper we set the ground for analysing such systems also in the presence of external forces. \\

 In this paper  we first calculate the $\Gamma$-limit $H$ of the energy $H_n$  of the discrete system  as the number of atoms $n$ tends to infinity (Section~\ref{sett}). As one can see from the obtained formula, any information on the number of cracks (i.e., the jump points of macroscopic deformation $u$) is lost in the $\Gamma$-limit: indeed the positive
singular part of the derivative of $u$ has no influence on $H$.
Therefore, in order to gain further insight in the limiting behaviour of the considered chain of atoms, we study a higher order description of $H_n$ 
by employing the development by $\Gamma$-convergence. More precisely, one considers the sequence of functionals 
\begin{gather}
H_{1,n}^\ell(u_n):=\frac{H_n(u_n)-\inf_u H(u)}{\lambda_n}
\label{DefinitionH1n}
\end{gather}
with the goal of determining a $\Gamma$-limit for $H_{1,n}^\ell$, denoted by $H^1$ and  called \emph{first order $\Gamma$-limit} or $\Gamma$-limit of first order of $H_n$.
Incorporation of forces poses additional challenges, as the value of the minimum of $H$ is not known explicitly and therefore the derivation of properties of $H_{1,n}^\ell$ requires a careful analysis.

 Here  we restrict our attention to the characterization of the minimizers of the zeroth-order $\Gamma$-limit and the compactness results relevant for the identification of the first order $\Gamma$-limit. The full study of the $\Gamma$-development of the discrete energies is postponed to a future work.
In order to derive compactness results, one is tempted to employ the method of even-odd
interpolation developed in \cite{BraidesCicalese} and used in the previous papers. However this fails without proper modifications:  the even-odd interpolation strongly modifies the deformation $u$ of the chain of atoms (while preserving the gradient) and therefore it cannot provide enough control of the external force that is depending explicitly on the variable $u$.\\
The novel method that we develop for the proof of the compactness results involves the construction of suitable competitors for $H$ that we denote by $v_{1,n}$ and $v_{2,n}$. The goal is to choose $v_{1,n}$ and $v_{2,n}$ such that the difference $\frac{1}{\lambda_n}(H_{n}(u_n)-\frac{1}{2}(H(v_{1,n})+H(v_{2,n})))$ provides control of $(u'_n - \gamma)_+$ while at the same time being controlled by $H_{1,n}^\ell(u_n)$. Here, $\gamma$ is the minimizer of the effective interaction potential and thus the equilibrium condition of $u'$ in the absence of external forces, cf.\ Assumption~$\mathscr{A}$ in Section~\ref{assumptions}. The definition of  $v_{1,n}$ and $v_{2,n}$ is based on a careful step-by-step construction that truncates the slopes of the standard interpolation and introduces jumps in suitably chosen points of the domain (see Section~\ref{proofcomp}).\\

%Part (a) of Theorem~\ref{Compactness1} shows that 
%sequences of configurations that keep the rescaled energies $H_{1,n}^\ell$ uniformly bounded have only a finite number of bonds such that $(u_n'-\gamma)_+ \geq \varepsilon$.
%We remark that in contrast to the previous works \cite{BraidesCicalese, ScardiaSchloemerkemperZanini} we do not expect the $L^1$ limit of an equibounded sequence to have derivative equal to $\gamma$ almost everywhere if $\ell \geq \gamma$. This is due to the effect of the external force applied on the system.
%Part (b) provides a more precise information on the magnitude of $(u_n-\gamma)_+$.
%In part (c) we prove that the ratio of compression of the material remains uniformly bounded along sequences for which the rescaled energies $H_{1,n}^\ell$ are bounded.  
%This is the asymptotic counterpart of Proposition~\ref{GradientBoundedBelow}. It ensures that in the derivation of the first order $\Gamma$-limit the singular behaviour of the potentials is in fact immaterial.

 The outline of the paper is as follows.  In Section~\ref{MinimizersOfLimit} we study properties for minimizers of the zeroth-order $\Gamma$-limit in order to  gain  a preliminary understanding of the first order $\Gamma$-limit. More precisely, we characterize the points of the domain (depending of the external force) where a non-elastic behaviour can occour. Then  we show further regularity results of the minimizers and we prove that there cannot be regions in the domain  with a complete compression.  This part is inspired by 
 \cite{BraidesDalmasoGarroni}, where the special case of a dead load $\Phi(x,u)=-f(x)u(x)$ is considered.  \\
In Section~\ref{seccomp} we state the main results concerning the compactness estimates for the first order $\Gamma$-limit: in Theorem~\ref{Compactness1}~(a) we prove that sequences of configurations that keep $H_{1,n}^\ell$ uniformly bounded have only a finite number of bonds such that $(u_n'-\gamma)_+ \geq \varepsilon$. Notice that this result is different to the usual compactness estimates obtained in previous related works as \cite{ BraidesCicalese,ScardiaSchloemerkemperZanini}. Indeed, the energy does not provide a control of the distance of $u_n'$ from $\gamma$, but only of its positive part. This is an effect due to the presence of the external forces. 
In Theorem~\ref{Compactness1} (b) we provide a more precise information on the magnitude of $(u'_n-\gamma)_+$. 
Finally, in Theorem~\ref{Compactness1} (c) we prove that the ratio of compression of the material remains uniformly bounded along sequences that keep the rescaled energies $H_{1,n}^\ell$ bounded.  
This ensures that in the derivation of the first order $\Gamma$-limit the singular behaviour of the potentials at zero is immaterial.\\
 The set of assumptions on the interaction potentials that are used throughout the paper (denoted by Assumption~$\mathscr{A}$) will be written explicitly only in Section~\ref{assumptions}. The reader can go through the statements of the main results assuming that the interaction potentials are two standard Lennard-Jones potentials (see Remark~\ref{rem}).\\
In Sections~\ref{proofproperties} and \ref{proofcomp} we provide the proofs of the results stated in Section~\ref{MinimizersOfLimit} (for the properties of minimizers of the limit functional) and Section~\ref{seccomp} (for the main results about compactness), which concludes the paper. \ae

%$****$ \\
%
%moreover, \footnote{Finally we could mention something about the difficulties for the liminf and limsup inequalities and say that this is future work.} the proof of the liminf inequality cannot be modified in a
%straightforward way since this would require showing fast convergence of any
%recovery sequence $(u_n)$ for the first-order $\Gamma$-limit to $u$ -- a problem
%which at first sight might seem easier than deriving estimates on the
%difference of the derivatives (which has been done in the previous works), but
%which in fact turns out to be impossible: our interaction potential fails to be
%coercive which makes a direct application of an inequality of Poincare type
%impossible; a simple consideration in the case of vanishing forces shows that
%the convergence of $(u_n)$ to $u$ can be arbitrarily slow (while for any recovery
%sequence for the first-order $\Gamma$-limit, the rate of convergence of the
%(heuristically) ``absolutely continuous part'' $u_n'$ of the
%distributional derivative of $u_n$ to the absolutely continuous part $u'$
%of the distributional derivative of $u$ can be controlled).
%
%\medskip

%{\bf Notation.} Throughout the paper, we use the following notations. By $C$, we denote a generic constant which value may change
%from line to line.
%
%
%

\section{Setting and main results}

\subsection{Setting}\label{sett}
A configuration of the chain of $n+1$ atoms
is described by a map
$u_n: \lambda_n\Z \cap [0,1] \rightarrow \mathbb{R}$, where we abbreviate $\lambda_n:=\frac{1}{n}$, $n\in\N$. 
As it is customary, we call the set of all possible configurations $\mathcal{A}_n(0,1)$ and we identify it with the set of all piecewise affine interpolations on $[0,1]$:
\begin{equation*}
\mathcal{A}_n(0,1) = \{u:[0,1] \rightarrow \R: u\in C([0,1]), u(x) \mbox{ is affine in }(i\lambda_n,(i+1)\lambda_n) \ \forall i=0,\ldots,n-1\}.
\end{equation*} 
We also denote by $u_n^i := u(i\lambda_n)$ the deformed
configuration of the $i$-th atom in the chain. 

The
atoms are assumed to interact with their nearest neighbours and next-to-nearest
neighbours; the interactions are described using two non-convex potentials $J_1$ and $J_2$
for nearest-neighbour and next-to-nearest neighbour interactions, respectively.
We furthermore assume that the external forces are described by a potential
$\Phi:[0,1]\times \mathbb{R}\rightarrow \mathbb{R}$, where the first variable
describes the position of the atom in the reference configuration and the
second variable describes the position of the atom in the deformed
configuration.

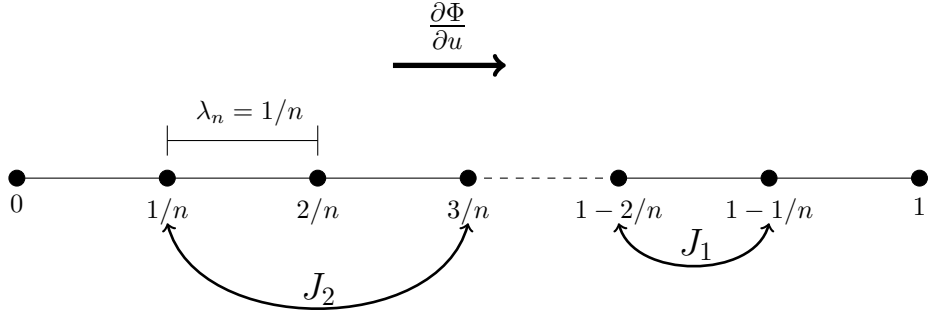
\begin{figure}
\begin{center}
\begin{tikzpicture}
\draw[->,line width=2pt] (3,1.5) to (4.5,1.5) node[above=15pt, left=10pt]{\Large $\frac{\partial\Phi}{\partial u}$};
\draw  (-2,0)  node[below=3pt]{0}  -- (4,0);
\draw  [dashed] (4,0)   -- (6,0);
\draw  (6,0)   -- (10,0);
\draw  (0,0.5) -- (2,0.5);
\draw  (0,0.3) -- (0,0.7);
\draw  (2,0.3)  -- (2,0.7) node[above=5pt, left=2pt]{$\lambda_n = 1/n$};
\filldraw (-2,0) circle (3pt); 
\filldraw (0,0) node[below=4pt]{$1/n$} circle (3pt);
\filldraw (2,0) node[below=4pt]{$2/n$} circle (3pt);
\filldraw (4,0) node[below=4pt]{$3/n$} circle (3pt);
\filldraw (6,0) node[below=4pt]{$1-2/n$} circle (3pt);
\filldraw (8,0) node[below=4pt]{$1 - 1/n$} circle (3pt);
\filldraw (10,0) node[below=4pt]{$1$} circle (3pt);
\draw [<->,line width=1pt] (0,-0.6) to [out=-75,in=-105] (4,-0.6) node[below=25pt, left=46pt]{\Large $J_2$};
\draw [<->,line width=1pt] (6,-0.6) to [out=-75,in=-105] (8,-0.6) node[below=8pt, left=17pt]{\Large $J_1$};

\end{tikzpicture}
\end{center}

\caption{A chain of $n+1$ atoms with interactions given by the potentials $J_1$ and $J_2$ and the external potential $\Phi$.
\label{chain}
}
\end{figure}

\begin{remark}\label{rem} 
Throughout the paper we assume the potential $\Phi$ to be in $C^2([0,1] \times \R)$. Moreover the assumptions on $J_j$ for $j=1,2$ are quite classical for the usual convex-concave interaction potentials; they are satisfied, e.g., by the standard Lennard-Jones potentials. In order to avoid technicalities in this first part of the paper we summarize the set of assumptions satisfied by $\Phi, J_1$ and $J_2$ in Assumption~$\mathscr{A}$. The precise assumptions will be written explicitly and commented on in Section~\ref{assumptions}. For the moment the reader can think of them as a standard set of assumptions satisfied by the Lennard-Jones potential.   The constant $\gamma>0$ denotes the minimizer of the effective potential $J_0$ defined in \eqref{J0def}, and $(0, \gamma^c+c)$ for some $c>0$ is the regime of strict convexity of $J_1$ and $J_2$, see [H0] in Assumption~$\mathscr{A}$ for details. \ae
\end{remark}

\begin{example}
(i) If the external force $f$ depends only on the Lagrangian coordinate, i.e., describes a dead load, the potential simply is
%The simplest example for a potential  includes the potentials $\Phi(x,u)$ that are generated by so called \emph{dead loads}, i.e. applied forces that depend only of the position $x$:
\begin{equation} \label{examplePhi}
\Phi(x,u(x))=-f(x)u(x)\,.
\end{equation}
%where $f$ denotes a sufficiently smooth function representing the applied force (see \cite{BraidesDalmasoGarroni} for the mathematical description of softening phenomena with dead loads).

(ii) If the external force $f$ also depends on the Eulerian coordinate, which models what is sometimes called live loads and is, e.g., of interest in the case of a chain of atoms that is placed on a substrate \cite{Wagneretal}, the potential reads
\begin{equation*}
\Phi(x,u(x))=-\int_0^{u(x)} f(x,w) \,dw\,.
\end{equation*}
\end{example}

We define the energy associated with a configuration $u_n$ by
\begin{align}\label{acca}
H_n(u_n):= &
\sum_{i=0}^{n-1}\lambda_n
J_1\left(\frac{u_n^{i+1}-u_n^{i}}{\lambda_n}\right)
+\sum_{i=0}^{n-2}\lambda_n
J_2\left(\frac{u_n^{i+2}-u_n^{i}}{2\lambda_n}\right)
+\sum_{i=0}^{n}\lambda_n \Phi(i\lambda_n,u_n^i)\,.
\end{align}

Moreover, we add Dirichlet boundary conditions in the following way:
\begin{itemize}
  \item[\textrm{[B0]}] \quad $u_n^0=0$ and $u_n^n=\ell$,
\end{itemize}
and we prescribe the slope of the discrete configuration at the boundary, i.e. we
require that
\begin{itemize}
  \item[\textrm{[B1]}] \quad $u_n^1=\lambda_n \theta_0$ and
  $u_n^{n-1}=\ell-\lambda_n \theta_1$
\end{itemize}
for some fixed $\theta_0,\theta_1>0$. 
The reader can compare assumptions [$B0$] and [$B1$] with \cite{BraidesCicalese} where only [$B0$] is assumed and
\cite{ScardiaSchloemerkemperZanini} where additionally [$B1$] is imposed. 
The Dirichlet conditions [$B0$] and [$B1$] correspond to the situation of a \emph{hard loading device}. As remarked in \cite{CharlotteTrusk2002,CharlotteTrusk2008} it is natural to impose four Dirichlet boundary conditions in the case of next-to-nearest neighborhood interactions as they ensure the equilibrium of the discrete system.

We set 
\begin{equation*}
\mathcal{A}^{\ell}_n(0,1) := \{u \in \mathcal{A}_n(0,1) : [B0] \mbox{ and } [B1] \mbox{ hold}\}
\end{equation*}

and 

\begin{equation}
H_n^\ell(u_n)=\left\{
\begin{array}{ll}
H_n(u_n) & \mbox{ if } u_n \in \mathcal{A}^{\ell}_n(0,1)\, , \\
+\infty & \mbox{ otherwise}\, .
\end{array}
\right.
\end{equation}

Endowing $\mathcal{A}^\ell_n(0,1)$ with the $L^1$ topology we define the (zeroth-order) $\Gamma$-limit of $H^\ell_n$ for $n\rightarrow +\infty$ as
\begin{gather*}
H:= \underset{n\rightarrow +\infty}{\Gamma-\lim}\, H^\ell_n
\end{gather*}
in the $L^1$ topology (see \cite{Braides} for an introduction to $\Gamma$-convergence).

In order to identify the $\Gamma$-limit $H$ we define the effective potential $J_0$ as the inf-convolution
of $J_1$ and $J_2$, i.e.
\begin{gather}\label{J0def} 
J_0(z):=\inf\left\{J_2(z) + \frac{1}{2}\left(J_1(z_1)+J_1(z_2)\right) : \frac{1}{2}(z_1 + z_2) = z   \right\}.
\end{gather}
Its minimizer is denoted by $\gamma$, see Assumption~$\mathscr{A}$. 
Further,  $J_0^{\ast\ast}$ is the lower convex envelope of $J_0$, cf.\ \eqref{J0starstar}. 
We also denote the set of all functions
$u\in BV((a-1,b+1))$ that are equal to 0 on $(a-1,a)$ and equal to $\ell>0$ on $(b,b+1)$ by $BV^\ell([a,b])$. Similarly we define the space $SBV^\ell([a,b])$. 
Moreover, given $u\in BV([a,b])$, the notation $D^s u$ refers to the singular part of the
distributional derivative of $u$, while $u'$ refers to the absolutely continuous part.
\begin{proposition}
\label{IdentificationGammaLimit}
Suppose that Assumption $\mathscr{A}$ is satisfied.
Then the $\Gamma$-limit of $(H^\ell_n)_n$ with respect to the $L^1(0,1)$ topology is given
by
\begin{gather}\label{gamma}
H(u):=
\begin{cases}
\displaystyle\int_{0}^1 J_0^{\ast\ast}(u')+\Phi(x,u)~dx
\quad \text{ if }  u\in BV^\ell({[0,1]})\text{ and
}D^s u\geq 0
\\
+\infty \quad \text{ else}\, .
\end{cases}
\end{gather}
\end{proposition}

The derivation of the zeroth-order $\Gamma$-limit for functionals of the type of \eqref{acca} is classical and it has been established in \cite{ScardiaSchloemerkemperZanini} in case
$\Phi\equiv 0$ (see also \cite[Theorem 4.2]{BraidesCicalese}
and \cite[Theorem 3.2]{BraidesGelli} for earlier related results). We refer to Section~\ref{proofproperties} for the proof. 

As one can see from \eqref{gamma}, any information on the number of cracks (i.e., the jump points of $u$)
 is lost in the zeroth-order $\Gamma$-limit: indeed the positive
singular part of the derivative of $u$ has no influence on $H$.
Therefore, in order to gain further insight in the limiting behaviour of the considered chain of atoms, we provide a higher order description of $H_n$ 
by employing the development by $\Gamma$-convergence introduced in \cite{AnzellottiBaldo} (see also \cite{BraidesCicalese, BraidesTruskinovsky}). More precisely, one considers the sequence of functionals 
\begin{gather*}
H_{1,n}^\ell(u_n):=\frac{H_n^\ell(u_n)-\inf_u H(u)}{\lambda_n}
%\label{DefinitionH1n}
\end{gather*}
with the goal of determining a $\Gamma$-limit for $H_{1,n}^\ell$ (denoted by $H^1$), called \emph{first order $\Gamma$-limit} of $H_n$.

The following properties of the minimizers of the limit functional will be useful to study the first order $\Gamma$-limit.

\subsection{Properties of minimizers of the limit functional}\label{MinimizersOfLimit}

The existence of minimizers of the limit functional is obtained by a classical application of the direct method in the calculus of variations (see, e.g. \cite{Dacorogna}). 
\begin{proposition}
\label{MinimizerExistence}
Suppose that Assumption $\mathscr{A}$ is satisfied.
Then the functional $H:L^1(0,1) \rightarrow (-\infty, + \infty]$ defined in \eqref{gamma}
has a minimizer $u \in BV^\ell([0,1])$.
\end{proposition}

The study of minimizers of the zeroth-order $\Gamma$-limit is fundamental for the identification of the first order $\Gamma$-limit.
Indeed, from \eqref{DefinitionH1n} it follows that the first-order $\Gamma$-limit is infinite for all $u$ in the domain of definition of $H$ with $H(u)-\inf_v H(v)>0$.
For this reason we devote the rest of this section to the study of the properties of minimizers of $H$.

%\footnote{AS: Since it is already written in the introduction, I deleted: %"Note that in the special case $\Phi(x,u)=-f(x)u(x)$ the analysis performed below
%has already been carried out by Braides, Dal Maso and Garroni in
%\cite{BraidesDalmasoGarroni}."} 

\medskip

Depending on the external potential $\Phi(x,u)$ we can identify a region in $[0,1]$ in which there is elastic behaviour and no
cracks. More precisely, we will show that
both $(u'-\gamma)_+$ and $D^s u$
necessarily vanish outside some set determined by the external potential $\Phi$. 

\begin{proposition}
\label{JumpCondition}
Suppose that Assumption $\mathscr{A}$ is satisfied.
Let $u \in BV^\ell([0,1])$ be a minimizer of $H$. Let $F:[0,1] \rightarrow \mathbb{R}$ be defined by
\begin{align*}
F(x):=\int_x^1 - \frac{\partial\Phi}{\partial u}(y,u(y)) ~dy\,.
\end{align*}
Let $M$ be the set of
(global) maximum points of $F$. Then the supports of $(u'-\gamma)_+$ and of $D^s u$ are subsets of $M$.
\end{proposition}
 The following examples show that indeed a minimizer can have jumps both in the domain $(0,1)$ and at the boundary, as well as a derivative that is strictly bigger than $\gamma$. The first example deals with the case of absence of external forces, while the second one is an instance of external forces depending only on Lagrangian coordinates, i.e., it deals with dead loads.
\begin{example}
Let $\Phi(x,w) = 0$ for every $x,w$. Then $M=[0,1]$ and it is clear from \eqref{J0starstar} that any function $u\in BV^\ell([0,1])$ with slope bigger than $\gamma$ and non-negative $D^su$ is a minimizer for $H$,    a case that was already handled in \cite{ScardiaSchloemerkemperZanini}.
\end{example}

\begin{example}\label{exedead}
 Let  $\Phi(x,u(x)) = - f(x)u(x)$. Thus $-\partial_u \Phi(x,u(x)) = f(x)$. Firstly, note that  there can only be a crack at $x_0\in (0,1)$ if $F$ in Proposition~\ref{JumpCondition} has a global maximum in $x_0$, i.e., if $f(x_0)=0$. Secondly, in the case $f(x) < 0$ for all $x\in (0,1)$ and $\ell>\gamma$, we have $M=\{1\}$. Assume that $u \in BV^{\ell}([0,1])$ is a minimizer of $H$. If we modify it by defining $\widetilde u \in BV^\ell([0,1])$ as
\begin{equation*}
\widetilde u(x) := \int_0^x \min\{u',\gamma\}\, dx\,,
\end{equation*}
then $J_0^{**} (\widetilde u') = J_0^{**} (u')$ and $\Phi(x,\widetilde u) \leq \Phi(x,u)$. Hence $H(\widetilde u) = H(u)$ and $1 \in S_{\widetilde u}$. 
\end{example}

The next proposition yields information regarding the behaviour of $\Phi$ at the jumps of a minimizer $u\in BV^\ell([0,1])$ of $H$. 

\begin{proposition}
\label{JumpCondition2}
Suppose that Assumption $\mathscr{A}$ is satisfied.
Let $u \in BV^\ell([0,1])$ be a minimizer of $H$. Then, given $x_0 \in S_u \cap (0,1)$, the condition
\begin{gather}\label{Jumpdiscrepancy}
\Phi(x_0,u(x_0-))=\Phi(x_0,u(x_0+))=\min_{w\in [u(x_0-),u(x_0+)]} \Phi(x_0,w)
\end{gather}
is satisfied. 
If $x_0 \in S_u\cap \{0,1\}$, then the following holds true:
\begin{align*}
\Phi(0,u(0+))&=\min_{w\in [u(0-),u(0+)]} \Phi(0,w)\,,
\\
\Phi(1,u(1-))&= \min_{w\in [u(1-),u(1+)]} \Phi(1,w)\,.
\end{align*}
\end{proposition}

Next we prove that the derivative of any minimizer of $H$ is bounded away from zero almost everywhere. Physically, this means that the ratio of compression of the material is bounded in the continuum limit.

\begin{proposition}
\label{GradientBoundedBelow}
Suppose that Assumption $\mathscr{A}$ is satisfied.
Then any minimizer $u\in BV^\ell([0,1])$ of $H$ satisfies $\essinf_{x\in [0,1]}u'(x)>0$.
\end{proposition}

We finally compute the Euler-Lagrange equation associated with $H$. The Euler-Lagrange equation yields continuity of $u'$ on a certain set.  

\begin{proposition}
\label{EulerLagrange}
Suppose that Assumption $\mathscr{A}$ is satisfied.
Any minimizer $u$ of $H$ is a solution to the
Euler-Lagrange-Equation
\begin{align}
\label{EulerLagrangeEq}
\int_0^1 {J_0^{\ast\ast}}'(u')\phi'+\frac{\partial\Phi}{\partial
u}(x,u)\phi ~dx=0 \quad \mbox{ for any } \phi\in C_0^\infty([0,1]).
\end{align}
Moreover $u'$ is continuous on the open set
\begin{align}
\label{SetSmallerGamma}
M_\gamma:= \{x\in [0,1]: u'(x)<\gamma\}\,.
\end{align}
\end{proposition}

Next we consider the compactness properties of the $\Gamma$-limit of first order. The proofs of the above propositions are given in Section~\ref{proofproperties}.

\subsection{Compactness}\label{seccomp}

In this section we state the main results of this paper: the compactness of sequences with bounded rescaled energy $H_{1,n}^\ell$.  As mentioned earlier, this is of central importance for the derivation of the first order $\Gamma$-limit.

\begin{theorem}
\label{Compactness1}
Suppose that Assumption $\mathscr{A}$ is satisfied. Let $(u_n) \subset \mathcal{A}^\ell_n(0,1)$ be a sequence of configurations such that $\sup_n H_{1,n}^\ell(u_n) < +\infty$. Then the following statements hold:\\
\begin{itemize}
\item[(a)]
For every $\varepsilon > 0$ there exists a constant $C=C(\varepsilon)>0$ independent of $n$ (but possibly depending on the sequence $(u_n)_n$) such that
\begin{equation}\label{expl}
\#\left\{ i:\frac{u_n^{i+1}-u_n^i}{\lambda_n}\geq\gamma+\varepsilon
\right\}\leq C(\varepsilon)
\end{equation}
and 
\begin{equation}\label{eq17-?}
\#\left\{ i:\left|\frac{u_n^{i+1}-u_n^i}{\lambda_n}
-\frac{u_n^{i+2}-u_n^{i+1}}{\lambda_n}\right|\geq\varepsilon \right\} \leq C(\varepsilon)\, .
\end{equation} 

\item[(b)]
Set $\mathcal{I}_n := \left\{i: \frac{u_n^{i+1} - u^i_n}{\lambda_n} \leq \gamma^c \right\}$. Then there exists $C>0$ such that
\begin{gather} \label{eq17-1}
\sum_{i\in \mathcal{I}_n} \left(\frac{u_n^{i+1} - u_n^i}{\lambda_n} - \gamma \right)_+^2 \leq C\,. 
\end{gather}

\item[(c)]
There holds
\begin{gather} \label{liminfslopebigger}
\liminf_{n\rightarrow +\infty} \min_{i\in \{0,\ldots,n-1\}}
\frac{u_n^{i+1}-u_n^i}{\lambda_n}>0\,.
\end{gather}

\end{itemize}
\end{theorem}

Some comments are in order.  \\
Part (a) of Theorem~\ref{Compactness1} shows that 
sequences of configurations that keep the rescaled energies $H_{1,n}^\ell$ uniformly bounded have only a finite number of bonds such that $(u_n'-\gamma)_+ \geq \varepsilon$.
We remark that in contrast to the previous works \cite{BraidesCicalese, ScardiaSchloemerkemperZanini} we do not expect the $L^1$ limit of an equibounded sequence to have derivative equal to $\gamma$ almost everywhere if $\ell \geq \gamma$. This is due to the effect of the external force applied on the system.
Part (b) provides a more precise information on the magnitude of $(u_n-\gamma)_+$.
In part (c) we prove that the ratio of compression of the material remains uniformly bounded along sequences for which the rescaled energies $H_{1,n}^\ell$ are bounded.  
This is the asymptotic counterpart of Proposition~\ref{GradientBoundedBelow}. It ensures that in the derivation of the first order $\Gamma$-limit the singular behaviour of the potentials is in fact immaterial.
The proofs of parts (a), (b) and (c) of Theorem~\ref{Compactness1} can be found in Sections~\ref{a}, \ref{b} and \ref{c}, respectively.

\medskip

As a consequence of Theorem~\ref{Compactness1} we deduce the following result about the convergence of a sequence of configurations equibounded in energy. The proof is an adaptation of Theorem 3.1 in \cite{BraidesGelli2002} (see also \cite{BraidesCicalese}); it can be found in Section~\ref{c}.

\begin{proposition}
\label{Compactness3}
Suppose that Assumption $\mathscr{A}$ is satisfied. Let $(u_n)_n \subset \mathcal{A}^\ell_n(0,1)$ be a sequence of configurations such that $\sup_{n\rightarrow +\infty}
H_{1,n}^\ell(u_n)<+\infty$. 
Then, up to a subsequence, $u_n\rightarrow u$ strongly in $L^1(0,1)$, where $u\in SBV^\ell([0,1])$ is such that
\begin{itemize}
\item[$(i)$] $\#S_u <+\infty$,
\item [$(ii)$] $[u] > 0$ in $S_u$,
\item [$(iii)$] $u' \leq \gamma$ almost everywhere in $(0,1)$.
%\item [$(iv)$] There exists $S \subset [0,1]$ such that $\#S < +\infty$ and $u_n \rightarrow u$ in $W_{loc}^{1,\infty}((0,1) \setminus S)$.
\end{itemize}   
Here, $[u]$ denotes the difference of the left and the right limits of $u$.
\end{proposition}

As a consequence of Proposition~\ref{Compactness3}, no information is encoded in the first order $\Gamma$-limit if all minimizers of $H$ have slope strictly bigger than $\gamma$ in some set of positive measure or have a non-zero Cantor part of the derivative. Due to the effect of the external force this can happen easily as the following example shows.

\begin{example} \label{ex-quadr}
Consider an arbitrary function $\widetilde w \in C^2(0,1)$ such that $\widetilde w(0) = 0$, $\widetilde w(1)=\ell$ and $\widetilde w'(x) > \gamma$ for every $x\in (0,1)$. Then, choosing $\Phi(x,u) = (u-\widetilde w)^2$, we obtain that the function $u=\widetilde w$ is the unique minimizer of $H$. Hence for this choice of the potential $\Phi$ all the minimizers of $H$ have slope bigger than $\gamma$ everywhere. From Theorem~\ref{Compactness3} (iii) it follows that the first order $\Gamma$-limit is infinite and so it does not provide additional information on the behaviour of the system. 
\end{example}

As a positive result we show that under some further assumptions on the external force, there always exists a minimizer of $H$ with derivative bounded by $\gamma$. The proof can be found in Section~\ref{c}. Examples of potentials satisfying these conditions are $\Phi(x,u) = (u-\widetilde w)_+^4$  or $\Phi(x,u)=(u-\tilde w)_-^4$ with $\widetilde{w}$ as in Example~\ref{ex-quadr}.
\begin{proposition}
\label{LimitationsCondition}
Suppose that Assumption $\mathscr{A}$ is satisfied.
Assume in addition that for any $x\in[0,1] $
\begin{equation}\label{mpf}
\sign \frac{\partial \Phi}{\partial u}(x,w)  \quad \mbox{ is independent of }  w\in\R  
\end{equation}
and that $\sign \frac{\partial\Phi}{\partial u}$ only changes
its value at finitely many points in $[0,1]$. Then there exists a minimizer $u \in BV^\ell([0,1])$
of $H$ with $u'(x)\leq \gamma$ for almost every $x\in (0,1)$ and $|D^c u|([0,1])=0$. 
\end{proposition}

\section{Assumption $\mathscr{A}$ on interaction potentials and external forces}\label{assumptions}

We say that Assumption~$\mathscr{A}$ is satisfied if the interaction potentials $J_1$ and $J_2$  as well as the effective potential $J_0$ defined in \eqref{J0def} satisfy conditions [H0]--[H5] and if the external force is described by a potential $\Phi$ satisfying condition [$\Phi1$] stated as follows:
\begin{enumerate}
  \item[\textrm{[H0]}](strict convexity on a bounded interval) There exists $\gamma^c \in (0,+\infty)$ and $c>0$ such that $J_j$ is strictly convex in $(0,\gamma^c + c)$ for $j=0,1$. 
  
  \item[\textrm{[H1]}](regularity) $J_j\in C^2(0,\infty)$ for $j=0,1,2$.
  \item[\textrm{[H2]}](uniqueness of minimal energy configurations) For any fixed $z\in (0,\gamma^c)$
  \begin{gather*}
  \min\left\{J_2(z) + \frac{1}{2}\left(J_1(z_1)+J_1(z_2)\right) : \frac{1}{2}(z_1 + z_2) = z   \right\}
  \end{gather*}
  is attained at exactly $z_1 = z_2 = z$.
  \item[\textrm{[H3]}](behaviour at infinity) $J_j(z)\rightarrow
  J_j(\infty)\in \mathbb{R}$ for $z\rightarrow +\infty$ for $j=0,1,2$.  
  \item[\textrm{[H4]}](structure of $J_0$) $J_0$ has a unique
  minimum point $\gamma<\gamma^c$ with $\inf_{z\in
  [\gamma^c,+\infty)}J_0(z)>J_0(\gamma)$.
  \item[\textrm{[H5]}] $J_j(z)=+\infty$ for $z\leq 0$ and
  $J_j(z)\rightarrow +\infty$ as $z\rightarrow 0$ for $j=1,2$.
\end{enumerate}
%\textcolor{red}{The condition [H0] is needed for $i=1$ in even/odd interpolation in Lemma
%\ref{Compactness1}}.
We remark that assumption [H0] about the strict convexity of $J_1$ up to $\gamma^c$ is needed in part (b) of Theorem~\ref{Compactness1}. The other hypotheses are classical in the context of one-dimensional, non-convex discrete to continuum theory (see for example \cite{BraidesCicalese,ScardiaSchloemerkemperZanini}).\\

Our assumption on the potential $\Phi$ is as follows:
\begin{enumerate}
  \item[\textrm{[$\Phi1$]}] $\Phi\in C^2([0,1]\times \mathbb{R})$.
\end{enumerate}

\begin{remark}
Assumptions [H0] and [H4] imply that
%\begin{equation*}
%J_0(z) = J_0^{\ast \ast}(z) \qquad \forall z \leq \gamma.
%\end{equation*}
%Indeed
\begin{equation} \label{J0starstar}
J_0^{\ast \ast}(z) := \left\{
\begin{array}{ll}
J_0(z) & z<\gamma\, ,\\
J_0(\gamma) & z \geq \gamma\,.
\end{array}
\right.
\end{equation}
%is a convex function such that $\widetilde J(z) \leq J_0(z)$ for every $z\in \R$.
\end{remark}
\begin{remark}[Lennard-Jones potentials]
The classical Lennard-Jones potentials satisfy the assumptions [H0]--[H5]. Indeed for given $c_1,c_2 >0$ we define 
\begin{equation*}
J_1(z) = \frac{c_1}{z^{12}} - \frac{c_2}{z^6} \quad \mbox{and} \quad J_2(z) = J_1(2z)
\end{equation*}
for $z>0$ and we extend them to $+\infty$ on $(-\infty,0]$. One can check that $J_1, J_2$ and $J_0$ satisfy [H0]--[H5] (see Remark 4.1 in \cite{ScardiaSchloemerkemperZanini} for the detailed computation).
\end{remark}
%\begin{remark}[Double Yukawa potential]
%Another example of a potential satisfying assumptions $[H1] - [H5]$ is the double Yukawa potential. It is defined for $a>b>0$ and $c > 0$ as
%\begin{equation*}
%J_1(z) = \frac{c}{z} \left(e^{-a(z-1)} - e^{-b(z-1)}\right) \quad \mbox{and} \quad J_2(z) = J_1(2z)
%\end{equation*}
%for $z>0$ and extended to $+\infty$ on $(-\infty,0]$. One can check that also in this case $J_1, J_2$ and $J_0$ satisfy $[H1] - [H5]$. 
%\end{remark}
%

\section{Proofs for Proposition~\ref{IdentificationGammaLimit} and Section~\ref{MinimizersOfLimit}}\label{proofproperties}

\begin{proof}[\textbf{Proof of Proposition~\ref{IdentificationGammaLimit}}]

As the result for $\Phi \equiv 0$ has already been proved in \cite{ScardiaSchloemerkemperZanini} it is enough to show the convergence of the force term for any converging sequence $(u_n)_n \subset \mathcal{A}^\ell_n(0,1)$.

Let $(u_n)_n \subset \mathcal{A}^\ell_n(0,1)$ be a sequence of discrete configurations converging to some limit
$u$ in $L^1(0,1)$.
We can suppose without loss of generality that $u'_n > 0$ for every $n \in \N$. Indeed if there exists a sequence $n_k$ such that $u_{n_k}'(x)\leq 0$ in an interval, then $H(u_{n_k}) = +\infty$ for every $k$ and $H(u) = +\infty$ thanks to assumption [H5].
Thus we have
\begin{align*}
\MoveEqLeft{ \left|\sum_{i=0}^n \lambda_n \Phi(i\lambda_n,u_n^i)
-\int_{0}^1 \Phi(x,u(x))~dx\right| }
\\
\leq&
\sum_{i=0}^{n-1}
\bigg|
\lambda_n \Phi(i\lambda_n,u_n(i\lambda_n))
-\int_{i\lambda_n}^{(i+1)\lambda_n} \Phi(y,u(y)) ~dy
\bigg|
+ \lambda_n \sup_{x\in {[0,1]},0\leq w\leq \ell} |\Phi(x,w)|
\\
\leq&
\sum_{i=0}^{n-1} \int_{i\lambda_n}^{(i+1)\lambda_n}
|\Phi(y,u(y)) - \Phi(i\lambda_n,u_n(i\lambda_n))|\, dy + C\lambda_n 
\\
\leq&
\sum_{i=0}^{n-1} \sup_{x\in {[0,1]},0\leq w\leq \ell} |\nabla \Phi(x,w)|
\left(\lambda_n^2 +\int_{i\lambda_n}^{(i+1)\lambda_n}|u(y)-u_n(i\lambda_n)|	\, dy\right) + C\lambda_n\\
\leq& C\sum_{i=0}^{n-1} \int_{i\lambda_n}^{(i+1)\lambda_n}|u(y)- u_n(y)|\, dy
+ C\sum_{i=0}^{n-1} \int_{i\lambda_n}^{(i+1)\lambda_n}|u_n(y)-u_n(i\lambda_n)|\, dy + C\lambda_n\,.
\end{align*}
Thus as $u_n \rightarrow u$ in $L^1(0,1)$ it is enough to prove that 
\begin{gather}\label{noosc}
\lim_{n\rightarrow +\infty}  \sum_{i=0}^{n-1} \int_{i\lambda_n}^{(i+1)\lambda_n}|u_n(y)-u_n(i\lambda_n)|\, dy = 0\,.
\end{gather}
Indeed by the fact that $u_n$ is increasing for every $n$ we have
\begin{align*}
\sum_{i=0}^{n-1} \int_{i\lambda_n}^{(i+1)\lambda_n}&|u_n(y)-u_n(i\lambda_n)|\, dy \leq \sum_{i=0}^{n-1} \int_{i\lambda_n}^{(i+1)\lambda_n}u_n\left(\left(i+1\right)\lambda_n\right)-u_n(i\lambda_n)\, dy 
=\lambda_n\ell
\end{align*}
that yields \eqref{noosc}.

\end{proof}

\begin{proof}[\textbf{Proof of Proposition~\ref{MinimizerExistence}}]
By Proposition~\ref{IdentificationGammaLimit}, $H$ is the $\Gamma$-limit of some functional with respect to the $L^1(0,1)$-topology. Hence it is lower semicontinuous in $L^1(0,1)$ (cf., \cite[Proposition~1.28]{Braides}). As the weak-$\ast$ convergence of $(u_k)_k$ towards $u$ in $BV([0,1])$ implies that $u_k
\rightarrow u$ in $L^1(0,1)$, $H$ is sequentially lower semicontinuous with respect to weak-$\ast$ convergence in $BV([0,1])$.

Moreover, given $(u_k)_k \subset BV^\ell([0,1])$ of $H$ satisfying $\sup_k H(u_k)<+\infty$ there holds that $u'_k> 0$ because otherwise we had $J_0^{**} (u'_k) = +\infty$ by [H5].  Hence $u_k$ is monotone increasing and bounded by $\ell$ for every $k\in \N$. Thus $||u_k||_{L^1(0,1)}\leq \int_0^1 \ell ~dx=\ell$ and $|D u_k| ([0,1])=\ell-0$. This implies $||u_k||_{BV([0,1])} \leq C$ uniformly in $k$.

By the direct method of the calculus of variations we thus get existence of a minimizer (see also \cite[Theorem~1.21]{Braides}).
\end{proof}

\begin{proof}[\textbf{Proof of Proposition~\ref{JumpCondition}}]
Assume that the conclusion does not hold. We then show that $u$ is not a minimizer
of $H$. Let $\lambda$ be the measure defined by
\begin{gather*}
\lambda(A):=\int_{A\cap [0,1]} (u'-\gamma)_+ dx+D^s u (A\cap [0,1])
\end{gather*}
for any Borel set $A\subset\mathbb{R}$. If the conclusion does not hold, then the support of $\lambda$ would not be contained in $M$. Choose a point $m \in M$ and set
\begin{gather*}
\widetilde{u}(x):=\int_0^x \min\{u',\gamma\} ~dy
+\lambda([0,1]) \cdot \chi_{\{x:x>m\}}(x)\,.
\end{gather*} 
Following our definition of $BV^\ell([0,1])$ and observing that $Du((-1,x]) = Du([0,x])$ for any $x\in [0,1]$ we obtain for the right continuous good representative of $u$, again denoted by $u$, see \cite[Theorem 3.28]{AmbrosioFuscoPallara}, that  
\begin{align}\label{fundthm}
u(x)=\int_0^x u'(y) ~dy+D^su([0,x]), \quad x\in [0,1] \, .
\end{align}
Hence, as $\int_0^x \min\{u',\gamma\}~dy - \int_0^x u' ~dy = - \int_0^x (u'-\gamma)_+~dy$, we infer that 
\begin{align*}
\widetilde u (x)-u(x)
=& -\lambda([0,x])
+\lambda([0,1]) \cdot \chi_{\{x:x>m\}}(x)\\
=&-\int_0^x d(\lambda-\lambda([0,1])\delta_m)\, ,
\end{align*}
where $\delta_m$ denotes the Dirac measure concentrated in the point $m$. 
Obviously, $\widetilde{u} \in BV^\ell([0,1])$. We have $u'=\widetilde{u}'$ outside
the set $\{u'>\gamma\}$, inside of which we have $\widetilde{u}'=\gamma$. Since by \eqref{J0starstar} $J_0^{\ast\ast}(v)=J_0^{\ast\ast}(\gamma)$ for $v\geq \gamma$, for $\mu\in [0,1]$ we calculate
\begin{gather*}
H(u)-H((1-\mu)u+\mu \widetilde{u})=\int_0^1
\Phi(x,u(x))-\Phi(x,(1-\mu)u(x)+\mu\widetilde{u}(x))
~dx\, .
\end{gather*}
We now show that this contradicts the assumption that $u$ is a minimizer of $H$. Note that we choose convex combinations of $u$ and $\widetilde{u}$ here to ensure that the jumps do not become negative and we are dealing with monotone increasing functions.
Dividing by $\mu$ and letting $\mu \rightarrow 0$, we obtain 
%by Lebesgue's
%convergence theorem (which is applicable since
%$\frac{\partial\Phi}{\partial u}$ is bounded and since both $u$ and
%$\widetilde{u}$ have finite $L^\infty([0,1])$-norm) and by Fubini's theorem 
\begin{align*}
\left.\frac{d}{d\mu}H((1-\mu)u+\mu\widetilde{u})\right|_{\mu=0}
& = \int_0^1 \frac{\partial\Phi}{\partial u}(x,u(x)) (\widetilde{u}(x)-u(x)) ~dx
\\&
= \int_0^1 \frac{\partial\Phi}{\partial u}(x,u(x)) \left( - \int_0^x
d(\lambda-\lambda([0,1]) \delta_m)(y) \right) ~dx
\\&
= \int_0^1 \left( \int_y^1 - \frac{\partial\Phi}{\partial u}(x,u(x)) ~dx \right)
~d(\lambda-\lambda([0,1]) \delta_m)(y)
\\&
= \int_0^1 F(y) ~d(\lambda-\lambda([0,1]) \delta_m)(y)
= \int_0^1 F(y) ~d\lambda(y)-\lambda([0,1]) F(m)\\&
= \int_0^1 F(y) - F(m)~d\lambda(y)\\&
= \int_{(\rm supp \lambda) \cap M} F(y) - F(m) ~d\lambda(y) + \int_{(\rm supp \lambda) \setminus  M} F(y) - F(m) ~d\lambda(y)\\&
= \int_{(\rm supp \lambda) \setminus  M} F(y) - F(m) ~d\lambda(y)\, .
\end{align*}
Since $m$ is a (global) maximizer of $F$, we have $F(x)\leq F(m)$ for any $x\in[0,1]$ and for $x\notin M$ we have $F(x)<F(m)$. Moreover, due to the continuity of $F$, we have that $M$ is closed.
Hence, the subset of $M$, it follows that
\begin{gather*}
\left.\frac{d}{d\mu}H((1-\mu)u+\mu\widetilde{u})\right|_{\mu=0}<0\, ,
\end{gather*}
%For $\mu>0$ small, we therefore have $H(u)>H((1-\mu)u+\mu \widetilde{u})$ 
which yields that $u$ is not a minimizer of $H$. 
\end{proof}

\begin{proof}[\textbf{Proof of Proposition~\ref{JumpCondition2}}]
We argue by contraposition. To this end, we ``split'' the jump into two smaller
jumps and move one of the jumps a little to the left or right. If $x_0=0$ or
$x_0=1$, we can only move in one direction, thereby explaining the weaker
assertions of the proposition at the boundary.

Suppose $\Phi(x_0,u(x_0-))>\Phi(x_0,w)$ for some $w\in (u(x_0-),u(x_0+)]$. The other case
can be handled similarly.

Let $\varepsilon>0$ be small and define $\widetilde{u}_\varepsilon(x) \in BV^\ell([0,1])$ as
\begin{gather*}
\widetilde{u}_\varepsilon(x):=
\begin{cases}
u(x) & \text{ for } x\notin [x_0-\varepsilon,x_0]\, , \\
u(x)+w-u(x_0-) & \text{ for } x\in[x_0-\varepsilon,x_0]\, .
\end{cases}
\end{gather*}
We have
\begin{gather}\label{unoo}
\int_0^1 J_0^{\ast\ast}(u')~dx
=\int_0^1 J_0^{\ast\ast}(\widetilde{u}'_\varepsilon)
~dx
\end{gather}
and
\begin{gather}\label{duee}
\int_0^1 \Phi(x,\widetilde{u}_\varepsilon(x))-\Phi(x,u(x))
~dx
=\int_{x_0-\varepsilon}^{x_0}\Phi(x,u(x)+w-u(x_0-))-\Phi(x,u(x))
~dx\, .
\end{gather}
As $x\nearrow x_0$, we have $u(x)
\rightarrow u(x_0-)$ and thus $u(x)+w-u(x_0-)\rightarrow w$. Hence by continuity of $\Phi$, \eqref{unoo} and \eqref{duee}, we obtain that
\begin{align*}
\lim_{\varepsilon\rightarrow 0}
\frac{H(\widetilde{u}_\varepsilon)-H(u)}{\varepsilon} 
=& \lim_{\varepsilon\to 0} \frac1\varepsilon \int_{x_0-\epsilon}^{x_0}\Phi(x,u(x)+w-u(x_0-))-\Phi(x,u(x))
~dx \\
=&  \Phi(x_0,w)-\Phi(x_0,u(x_0-))<0\, .
\end{align*}
Hence $u$ is not a minimizer of $H$, which finishes the proof. 
\end{proof}

\begin{proof}[\textbf{Proof of Proposition~\ref{GradientBoundedBelow}}]

Let $u$ be a minimizer of $H$. Suppose by contradiction that $u$ satisfies  $\essinf_{x\in [0,1]}u'(x)=0$. Fix $m, n\in \N$ such that $m<n$ and set
\begin{align*}
A_n:=\left\{x\in [0,1]:u'(x)<\frac{1}{n}\right\}
\end{align*}
and
\begin{align*}
C_m:=\left\{x\in [0,1]:u'(x)>\frac{1}{m}\right\}
 .
\end{align*}
As $\essinf_{x\in [0,1]}u'(x)=0$, for every $n \in \mathbb{N}$ one has that $|A_n| > 0$. Moreover
\begin{equation}\label{convzero}
\lim_{n\rightarrow +\infty} |A_n| = 0\, ,
\end{equation}
as otherwise $H(u) = +\infty$. For the same reason
there exists $m_0>0$ such that for all $m>m_0$ we have $|C_m|>0$.
Next we define a regularization $\widetilde{u}$ of $u$, which 
increases the derivative of $u$ to $1/m$ where it becomes too small and decreases the
derivative on $C_m$ in order to meet the boundary condition at $1$. For $x\in[0,1]$ we set
\begin{align*}
\widetilde{u}(x):=
\int_{[0,x]\setminus A_n} & u'(y) ~dy
+D^s u([0,x])
+\int_{[0,x] \cap A_n} \frac{1}{m}\, dy
+\frac{|C_m\cap[0,x]|}{|C_m|}\int_{ A_n} \left( u'(y)-\frac{1}{m}\right)\, dy\,.
\end{align*}
It turns out that $\widetilde  u \in BV^\ell([0,1])$.
Notice in addition that 
\begin{equation}\label{fuori}
\widetilde{u}'(x) = u'(x) \quad \mbox{for a.e.   } x\in [0,1] \setminus (C_m \cup A_n)\, .
\end{equation}
% We set
% \begin{gather*}
% W:=\int_I \left(u'-\frac{1}{m}\right)_- \chi_{A_n} dy
% \end{gather*}
By the above definition of $\widetilde u$ and Equation \eqref{fundthm}, we obtain for any $x\in [0,1]$
\begin{align*}
|\widetilde u(x)-u(x)| \leq \int_{[0,x]\cap A_n} \Big|\frac{1}{m}-u'(y) \Big| \,dy + \int_{A_n} \Big|\frac{1}{m}-u'(y)\Big| dy
\frac{|C_m\cap[0,x]|}{|C_m|}\,.
\end{align*}
By the definition of $A_n$ and the fact that $u' \geq 0$, the following estimate holds:
\begin{gather}\label{difference}
\sup_{x\in [0,1]}|u(x)-\widetilde{u}(x)|\leq \frac{2|A_n|}{m}\, .
\end{gather}
Moreover, by [H1] and [H3] we know that $J_0^{\ast\ast}$ is Lipschitz continuous on
$[\frac{1}{2m},+\infty)$ for every $m>0$. Thus, using \eqref{convzero}, for any $m$ large enough there exists $n_0(m)$
such that for all $n\geq n_0$ for a.e. $x\in C_m$ the inequality
\begin{align*}
|J_0^{\ast\ast}(\widetilde{u}'(x)) - J_0^{\ast\ast}(u'(x))|
\leq C(m) \frac{|A_n|}{m}
\end{align*}
holds.

This yields for any $n$ large enough, using [$\Phi$1] as well as Equations \eqref{fuori} and \eqref{difference},
\begin{align*}
&H(\widetilde{u}) - H(u) \\
&= \int_{A_n} J_0^{\ast\ast}\left(\frac{1}{m}\right)
-J_0^{\ast\ast}(u') ~dx
+ \int_{C_m} J_0^{\ast\ast}\left(\widetilde{u}'\right)
-J_0^{\ast\ast}(u') ~dx
+\int_0^1 \Phi(x,\widetilde{u}(x))-\Phi(x,u(x)) ~dx
\\
& \leq
\int_{A_n}
J_0^{\ast\ast}\left(\frac{1}{m}\right)
-J_0^{\ast\ast}\left(\frac{1}{n}\right) ~dx
+\int_{C_m} C(m) \frac{|A_n|}{m} ~dx
+\frac{C|A_n|}{m}
\\
& \leq
|A_n| \left(J_0^{\ast\ast}\left(\frac{1}{m}\right)
-J_0^{\ast\ast}\left(\frac{1}{n}\right)+\frac{C(m)}{m}+\frac{C}{m}\right)
.
\end{align*}
Selecting $n$ large enough we see that, thanks to hypothesis [H5], the right hand side becomes negative. Hence we have reached a contradiction because $u$ minimizes $H$.
\end{proof}

\begin{proof}[\textbf{Proof of Proposition~\ref{EulerLagrange}}]

Let $u$ be a minimizer of $H$. For every test function $\phi\in C^\infty_0([0,1])$ consider $u+\lambda \phi$ for $\lambda \in \R$. Thanks to the minimality of $u$ we have
\begin{equation}\label{rot}
\frac{H(u+\lambda \phi) - H(u)}{\lambda} \geq 0\, .
\end{equation}
By Proposition~\ref{GradientBoundedBelow}, for $\lambda$ small enough, $ u'(x) + \lambda  \phi'(x)  > 0$ for almost every $x\in [0,1]$. Therefore, letting $\lambda \rightarrow 0$ in \eqref{rot} and using assumptions [$\Phi 1$], [H1] and [H3] to differentiate under the integral sign, we obtain that
\begin{align*}
\int_0^1 {J_0^{\ast\ast}}'(u')\phi'+\frac{\partial\Phi}{\partial
u}(x,u)\phi ~dx\geq 0\, .
\end{align*}
Then, replacing $\phi$ with $-\phi$, we infer the opposite inequality.

To prove the regularity of $u'$ in $M_\gamma$, defined in \eqref{SetSmallerGamma}, we notice that $u$ is bounded as it is a one-dimensional BV function. Hence, 
 $\frac{\partial\Phi}{\partial u}(x,u)$ is bounded as well. The Euler-Lagrange equation \eqref{EulerLagrangeEq} therefore entails that ${J_0^{\ast\ast}}'(u')$ is Lipschitz.  Notice that the preimage of $\{0\}$ of the continuous map ${J_0^{\ast\ast}}'(u')$ equals $\{ x \in [0,1] : u'(x) \geq \gamma\}$ by the strict monotonicity of ${J_0^{\ast\ast}}'(z)$ for $z\leq\gamma$ (assumption [H0]) and the fact that ${J_0^{\ast\ast}}'(z)$ is constant for $z\geq \gamma$. Since $\{0\}$ is closed, $\{ x \in [0,1] : u'(x) \geq \gamma\}$ is closed. Hence $M_\gamma$ is open.

 By the continuity of ${J_0^{\ast\ast}}'(u')$, the strict monotonicity of ${J_0^{\ast\ast}}'$ also implies the continuity of $u'$ on $M_\gamma$. 
\end{proof}

\section{Proofs for Section~\ref{seccomp}}\label{proofcomp}

In what follows we use extensively the following quantity:
\begin{gather} \label{Rdef}
R(z_1,z_2):=\frac{1}{2}\left[J_1(z_1)+J_1(z_2)\right]
+J_2\left(\frac{z_1+z_2}{2}\right)-J_0\left(\frac{z_1+z_2}{2}\right).
\end{gather}

First of all we propose a technical lemma that shows that under the assumption of convexity of $J_0$ and $J_1$ (see hypothesis [H0]), the functional $R(z_1,z_2)$ is bounded from below quadratically.
We will employ this estimate in the proof of part (b) of Theorem~\ref{Compactness1}.
\begin{lemma}\label{Jonecomp}
Let [H0]--[H5] be satisfied. Then
\begin{align}
R(z_1,z_2)&\geq 
c\left|z_1-z_2\right|^2
 \mbox{ for all } 0< z_1,z_2<\gamma^c\, . \label{Resti1}
\end{align}
\end{lemma}
\begin{proof}
By strict convexity of $J_1$ on
$(0,\gamma^c+c)$ (see assumption [H0]), the function $R(a+b,a-b)-C |b|^2$ is
still convex in $b$ for $C>0$ small enough as long as $b$ is such that $a+b,a-b<\gamma^c$. 
By the definition of $J_0$ we know that $R(a+b,a-b) \geq 0$. Moreover thanks to hypothesis [H2], $R(a,a) = 0$  whenever $a < \gamma^c$. Hence
\begin{displaymath}
R(a+b,a-b) = R(a+b,a-b) - R(a,a) \geq C|b|^2
\end{displaymath}
for $a+b,a-b<\gamma^c$. Hence \eqref{Resti1} follows.
\end{proof}
\ae

We are in position to prove the main theorem. In the next sections we  prove separately part (a),  (b) and (c) of Theorem~\ref{Compactness1}.

\subsection{Proof of part (a) of Theorem~\ref{Compactness1}}\label{a}
\begin{proof}

Let $(u_n) \subset \mathcal{A}^\ell_n(0,1)$ be a sequence of configurations such that $\sup_n H_{1,n}^\ell(u_n) < +\infty$. The first step of the proof of the compactness result for the $\Gamma$-limit of first order is the novel construction of suitable competitors for $H$ that allow to obtain the estimates in \eqref{expl} and \eqref{eq17-?}. In particular the goal is to define competitors $v_{1,n}$ and $v_{2,n}$ in such a way that the difference $\frac{1}{\lambda_n}\left(H_{n}^\ell(u_n)-\frac{1}{2}(H(v_{1,n})+H(v_{2,n}))\right)$ provides control of the quantities in \eqref{expl} and \eqref{eq17-?}. Observe in addition that this difference is controlled from above by $H_{1,n}^\ell(u_n)$.  We remark that the competitors $v_{1,n}$ and $v_{2,n}$ are not continuous and hence do not belong to $\mathcal{A}^\ell_n(0,1)$; however, they are admissible functions for $H$.

We will divide the proof in two steps: the first one is devoted to the construction of the competitors and the second one to the compactness estimates.
\vspace*{2mm}\\
\textbf{Step 1. Construction of the competitors $v_{1,n}$ and $v_{2,n}$}
\vspace*{2mm}\\
Define first ${\widetilde v}_{1,n}:[0,1]\rightarrow \mathbb{R}$ on the even atoms (and on the $n$-th one) in the following way:
\begin{displaymath}
\begin{array}{ll}
\widetilde v_{1,n}(2k\lambda_n):=u_n(2k\lambda_n) & \mbox{for } k\in \N_0, \  k\leq \frac{n}{2}\, ,\\
\widetilde v_{1,n}(n\lambda_n):=u_n(n\lambda_n)\, ,
\end{array}
\end{displaymath}
and then by piecewise affine interpolation for
all other values of $x\in [0,1]$. Define $\widetilde v_{2,n}$ similarly, but prescribing
the values of $\widetilde v_{2,n}$ at $0$, $(2k+1)\lambda_n$ and $n\lambda_n$
instead, for $k\leq \tfrac{n-1}2$.

As usual in this context, the general idea is to use the quantity
\begin{align} 
\nonumber
\MoveEqLeft{ \frac{1}{2}J_1\left(\frac{u_n^{i+1}-u_n^i}{\lambda_n}\right)
+\frac{1}{2}J_1\left(\frac{u_n^{i+2}-u_n^{i+1}}{\lambda_n}\right)
+J_2\left(\frac{u_n^{i+2}-u_n^{i}}{2\lambda_n}\right) }
\\&
=
J_0\left(\frac{u_n^{i+2}-u_n^{i}}{2\lambda_n}\right)
+R\left(\frac{u_n^{i+1}-u_n^{i}}{\lambda_n},\frac{u_n^{i+2}-u_n^{i+1}}{\lambda_n}\right) \label{utilize}
\end{align}
(see \eqref{Rdef} for the definition of $R$)
in order to compare $H_n(u_n)$ with
$\frac{1}{2}H(\widetilde v_{1,n})+\frac{1}{2}H(\widetilde v_{2,n})$. However, due to the
presence of a force term $\Phi$ this does not work directly, as the affine
interpolation involved in the definition of $\widetilde v_{i,n}$ may lead to a
large difference between $u_n$ and $\widetilde v_{i,n}$ at some points (see Fig.~\ref{EvenOddInterpolation}). We thus need to modify the functions $\widetilde v_{i,n}$ accordingly.

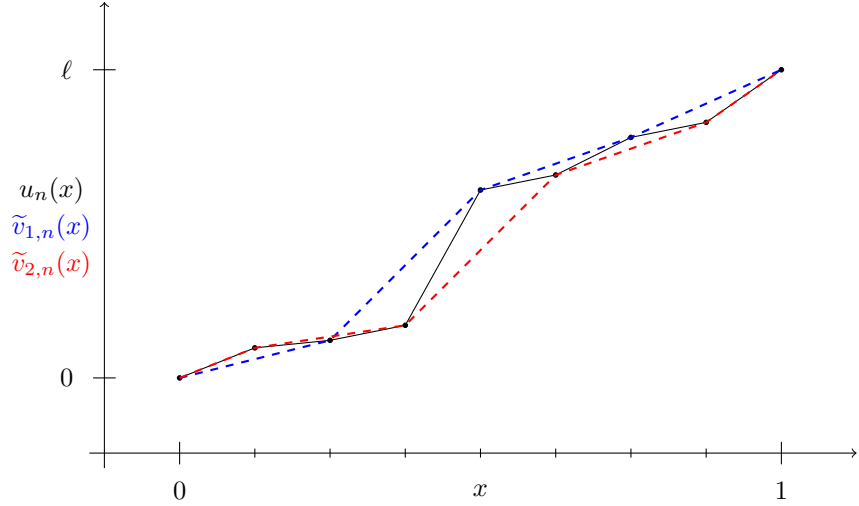
\begin{figure}
\begin{center}
\begin{tikzpicture}
\draw[->] (-1.2,0) -- (9,0);
\draw[->] (-1,-0.2) -- (-1,6);

\draw (0,-0.15) -- (0,0.15);

\draw (1,-0.06) -- (1,0.06);
\draw (2,-0.06) -- (2,0.06);
\draw (3,-0.06) -- (3,0.06);
\draw (4,-0.06) -- (4,0.06);
\draw (5,-0.06) -- (5,0.06);
\draw (6,-0.06) -- (6,0.06);
\draw (7,-0.06) -- (7,0.06);
\draw (8,-0.15) -- (8,0.15);

\draw (-1.15,1) -- (-0.85,1);
\draw (-1.15,5.1) -- (-0.85,5.1);

\node at (0,-0.5) {$0$};
\node at (8,-0.5) {$1$};
\node at (-1.5,1) {$0$};
\node at (-1.5,5.1) {$\ell$};
\node at (4,-0.5) {$x$};
\node at (-1.7,3.5) {$u_n(x)$};
\node[color=blue] at (-1.7,3.0) {$\widetilde{v}_{1,n}(x)$};
\node[color=red] at (-1.7,2.5) {$\widetilde{v}_{2,n}(x)$};

\draw (0,1) -- (1,1.4) -- (2,1.5) -- (3,1.7) -- (4,3.5)
-- (5,3.7) -- (6,4.2) -- (7,4.4) -- (8,5.1);

\draw[fill] (0,1) circle (0.03cm);
\draw[fill] (1,1.4) circle (0.03cm);
\draw[fill] (2,1.5) circle (0.03cm);
\draw[fill] (3,1.7) circle (0.03cm);
\draw[fill] (4,3.5) circle (0.03cm);
\draw[fill] (5,3.7) circle (0.03cm);
\draw[fill] (6,4.2) circle (0.03cm);
\draw[fill] (7,4.4) circle (0.03cm);
\draw[fill] (8,5.1) circle (0.03cm);

\draw[dashed,blue,thick] (0,1) -- (2,1.5) -- (4,3.5) -- (6,4.2) -- (8,5.1);
\draw[dashed,red,thick] (0,1) -- (1,1.4) -- (3,1.7) -- (5,3.7) -- (7,4.4) --
(8,5.1);

\end{tikzpicture}
\end{center}

\caption{A configuration of the discrete chain (black) and the
corresponding naive even-odd interpolators ${\widetilde v}_{1,n}$, ${\widetilde
v}_{2,n}$ (dashed lines, blue, red)
\label{EvenOddInterpolation}
}
\end{figure}

We define $\overline{v}_{1,n}$ to be equal to $\widetilde v_{1,n}$ everywhere, with the following
exception: if the slope of $\widetilde v_{1,n}$ exceeds $\gamma$ on an interval
$(2k\lambda_n,(2k+2)\lambda_n)$, where $\gamma$ is as in [H4], we prescribe $\overline{v}_{1,n}$ on the atoms $2k\lambda_n$, $(2k+1)\lambda_n$ and $(2k + 2) \lambda_n$ as 
\begin{align*}
\overline{v}_{1,n}(2k\lambda_n)& :=\widetilde v_{1,n}(2k\lambda_n)\, ,\\
\overline{v}_{1,n}((2k+1)\lambda_n)&:=u_n((2k+1)\lambda_n)\, , \\ %$, and
\overline{v}_{1,n}((2k+2)\lambda_n)&:=\widetilde v_{1,n}((2k+2)\lambda_n)\, .
\end{align*} 
Then, we extend it to $(2k\lambda_n,(2k+2)\lambda_n)$ not by affine interpolation but we instead impose $\overline{v}_{1,n}$ to
have slope $\gamma$ almost everywhere on the interval $(2k\lambda_n,(2k+2)\lambda_n)$. In this way we are forced to introduce jumps in some intermediate points of the interval $(2k\lambda_n,(2k+2)\lambda_n)$; in particular we allow for jumps in $(2k+\frac{1}{2})\lambda_n$ and $(2k+\frac{3}{2})\lambda_n$ (cf. Fig.~\ref{AdjustedInterpolation}).

Finally, we possibly perform a translation to $\overline{v}_{1,n}$ obtaining as a final outcome the function $v_{1,n}$:
if the jumps at $(2k+\frac{1}{2})\lambda_n$ or at $(2k+\frac{3}{2})\lambda_n$ are negative,
we add a constant to $\overline{v}_{1,n}$ on the interval
$((2k+\frac{1}{2})\lambda_n,(2k+\frac{3}{2})\lambda_n)$ in such a way that the
negative jump is eliminated and the function becomes continuous at that point
again (cf.\ Fig.~\ref{EvenOddInterpolation-3} and \ref{EvenOddInterpolation-4}). This does not cause another negative jump to appear on the other end of
the interval, since we are in the case of the slope of $\overline{v}_{1,n}$ being larger than $\gamma$ on
$(2k\lambda_n,(2k+2)\lambda_n)$.

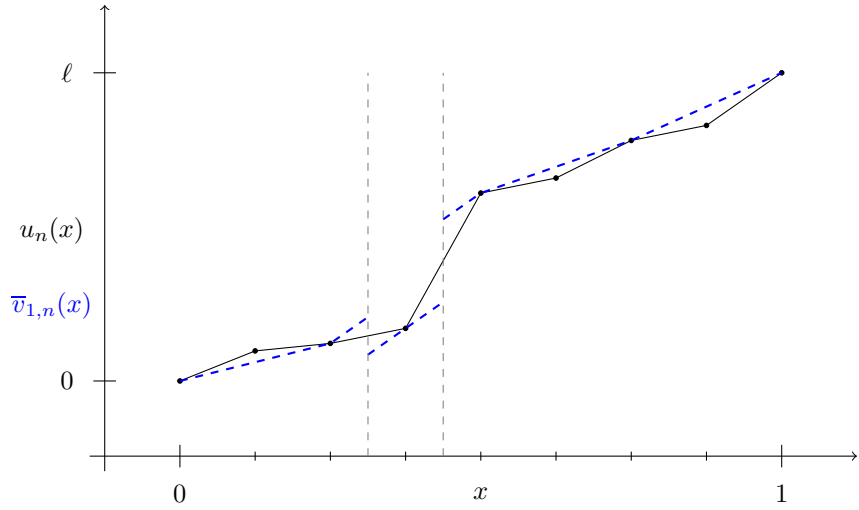
\begin{figure}
\begin{center}
\begin{tikzpicture}
\draw[->] (-1.2,0) -- (9,0);
\draw[->] (-1,-0.2) -- (-1,6);

\draw (0,-0.15) -- (0,0.15);
\draw (8,-0.15) -- (8,0.15);

\draw (-1.15,1) -- (-0.85,1);
\draw (-1.15,5.1) -- (-0.85,5.1);

\draw (1,-0.06) -- (1,0.06);
\draw (2,-0.06) -- (2,0.06);
\draw (3,-0.06) -- (3,0.06);
\draw (4,-0.06) -- (4,0.06);
\draw (5,-0.06) -- (5,0.06);
\draw (6,-0.06) -- (6,0.06);
\draw (7,-0.06) -- (7,0.06);

\node at (0,-0.5) {$0$};
\node at (8,-0.5) {$1$};
\node at (-1.5,1) {$0$};
\node at (-1.5,5.1) {$\ell$};
\node at (4,-0.5) {$x$};
\node at (-1.7,3.0) {$u_n(x)$};

\draw (0,1) -- (1,1.4) -- (2,1.5) -- (3,1.7) -- (4,3.5)
-- (5,3.7) -- (6,4.2) -- (7,4.4) -- (8,5.1);

\draw[fill] (0,1) circle (0.03cm);
\draw[fill] (1,1.4) circle (0.03cm);
\draw[fill] (2,1.5) circle (0.03cm);
\draw[fill] (3,1.7) circle (0.03cm);
\draw[fill] (4,3.5) circle (0.03cm);
\draw[fill] (5,3.7) circle (0.03cm);
\draw[fill] (6,4.2) circle (0.03cm);
\draw[fill] (7,4.4) circle (0.03cm);
\draw[fill] (8,5.1) circle (0.03cm);

\node[color=blue] at (-1.7,2.0) {$\overline{v}_{1,n}(x)$};

\draw[dashed,blue,thick] (0,1) -- (2,1.5) -- (2.5,1.85);
\draw[dashed,blue,thick] (2.5,1.35) -- (3.5,2.05);
\draw[dashed,blue,thick] (3.5,3.15) -- (4,3.5) -- (6,4.2) -- (8,5.1);

\draw[dashed,color=gray] (2.5,0) -- (2.5,5.1);
\draw[dashed,color=gray] (3.5,0) -- (3.5,5.1);
\end{tikzpicture}
\end{center}

\caption{A configuration of the discrete chain (black) and the
corresponding adjusted even-odd interpolating function $\overline{v}_{1,n}$ before the
final translation step (dashed line, bl)
\label{AdjustedInterpolation}}
\end{figure}

\begin{figure}
\begin{center}
\begin{tikzpicture}
\draw[->] (-1.2,0) -- (9,0);
\draw[->] (-1,-0.2) -- (-1,6);

\draw (0,-0.15) -- (0,0.15);
\draw (8,-0.15) -- (8,0.15);

\draw (-1.15,1) -- (-0.85,1);
\draw (-1.15,5.1) -- (-0.85,5.1);

\draw (1,-0.06) -- (1,0.06);
\draw (2,-0.06) -- (2,0.06);
\draw (3,-0.06) -- (3,0.06);
\draw (4,-0.06) -- (4,0.06);
\draw (5,-0.06) -- (5,0.06);
\draw (6,-0.06) -- (6,0.06);
\draw (7,-0.06) -- (7,0.06);

\node at (0,-0.5) {$0$};
\node at (8,-0.5) {$1$};
\node at (-1.5,1) {$0$};
\node at (-1.5,5.1) {$\ell$};
\node at (4,-0.5) {$x$};
\node at (-1.7,3.25) {$u_n(x)$};
\node[color=blue] at (-1.7,2.75) {$v_{1,n}(x)$};

\draw (0,1) -- (1,1.4) -- (2,1.5) -- (3,1.7) -- (4,3.5)
-- (5,3.7) -- (6,4.2) -- (7,4.4) -- (8,5.1);

\draw[fill] (0,1) circle (0.03cm);
\draw[fill] (1,1.4) circle (0.03cm);
\draw[fill] (2,1.5) circle (0.03cm);
\draw[fill] (3,1.7) circle (0.03cm);
\draw[fill] (4,3.5) circle (0.03cm);
\draw[fill] (5,3.7) circle (0.03cm);
\draw[fill] (6,4.2) circle (0.03cm);
\draw[fill] (7,4.4) circle (0.03cm);
\draw[fill] (8,5.1) circle (0.03cm);

\draw[dashed,blue,thick] (0,1) -- (2,1.5) -- (2.5,1.85);
\draw[dashed,blue,thick] (2.5,1.85) -- (3.5,2.55);
\draw[dashed,blue,thick] (3.5,3.15) -- (4,3.5) -- (6,4.2) -- (8,5.1);

\draw[dashed,color=gray] (2.5,0) -- (2.5,5.1);
\draw[dashed,color=gray] (3.5,0) -- (3.5,5.1);
\draw[dashed,color=gray] (0,1.85) -- (8,1.85);
\draw[dashed,color=gray] (0,2.55) -- (8,2.55);
\draw[dashed,color=gray] (0,3.15) -- (8,3.15);
\end{tikzpicture}
\end{center}

\caption{A configuration of the discrete chain (black) and the
corresponding adjusted even-odd interpolating function $v_{1,n}$ after the final
translation step (dashed line, blue); the left discontinuity has been fixed
during the translation step
\label{EvenOddInterpolation-3}}
\end{figure}
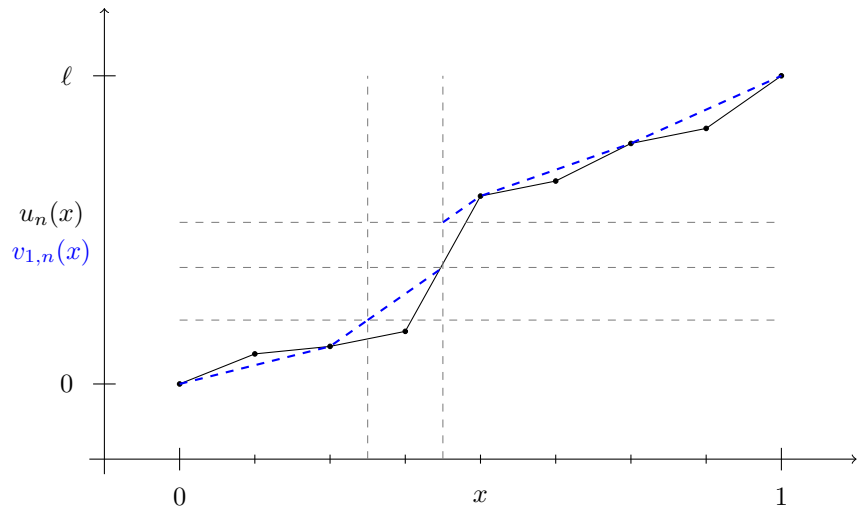

\begin{figure}
\begin{center}
\begin{tikzpicture}
\draw[->] (-1.2,0) -- (9,0);
\draw[->] (-1,-0.2) -- (-1,6);

\draw (0,-0.15) -- (0,0.15);
\draw (8,-0.15) -- (8,0.15);

\draw (-1.15,1) -- (-0.85,1);
\draw (-1.15,5.1) -- (-0.85,5.1);

\draw (1,-0.06) -- (1,0.06);
\draw (2,-0.06) -- (2,0.06);
\draw (3,-0.06) -- (3,0.06);
\draw (4,-0.06) -- (4,0.06);
\draw (5,-0.06) -- (5,0.06);
\draw (6,-0.06) -- (6,0.06);
\draw (7,-0.06) -- (7,0.06);

\node at (0,-0.5) {$0$};
\node at (8,-0.5) {$1$};
\node at (-1.5,1) {$0$};
\node at (-1.5,5.1) {$\ell$};
\node at (4,-0.5) {$x$};
\node at (-1.7,3.25) {$u_n(x)$};
\node[color=blue] at (-1.7,2.75) {$v_{1,n}(x)$};

\draw (0,1) -- (1,1.4) -- (2,1.5) -- (3,2.4) -- (4,3.5)
-- (5,3.7) -- (6,4.2) -- (7,4.4) -- (8,5.1);

\draw[fill] (0,1) circle (0.03cm);
\draw[fill] (1,1.4) circle (0.03cm);
\draw[fill] (2,1.5) circle (0.03cm);
\draw[fill] (3,2.4) circle (0.03cm);
\draw[fill] (4,3.5) circle (0.03cm);
\draw[fill] (5,3.7) circle (0.03cm);
\draw[fill] (6,4.2) circle (0.03cm);
\draw[fill] (7,4.4) circle (0.03cm);
\draw[fill] (8,5.1) circle (0.03cm);

\draw[dashed,blue,thick] (0,1) -- (2,1.5) -- (2.5,1.85);
\draw[dashed,blue,thick] (2.5,2.05) -- (3.5,2.75);
\draw[dashed,blue,thick] (3.5,3.15) -- (4,3.5) -- (6,4.2) -- (8,5.1);

\draw[dashed,color=gray] (2.5,0) -- (2.5,5.1);
\draw[dashed,color=gray] (3.5,0) -- (3.5,5.1);
\draw[dashed,color=gray] (0,1.85) -- (8,1.85);
\draw[dashed,color=gray] (0,2.05) -- (8,2.05);
\draw[dashed,color=gray] (0,2.75) -- (8,2.75);
\draw[dashed,color=gray] (0,3.15) -- (8,3.15);
\end{tikzpicture}
\end{center}

\caption{A slightly different configuration of the discrete chain
(black) and the corresponding adjusted even-odd interpolating function
$v_{1,n}$ (dashed line, blue); here no translation was required as no
``backward'' jump occurred as a result of the modification of $\widetilde v_{1,n}$
and consequently $v_{1,n}$ remains discontinuous at two points.
\label{EvenOddInterpolation-4}}
\end{figure}
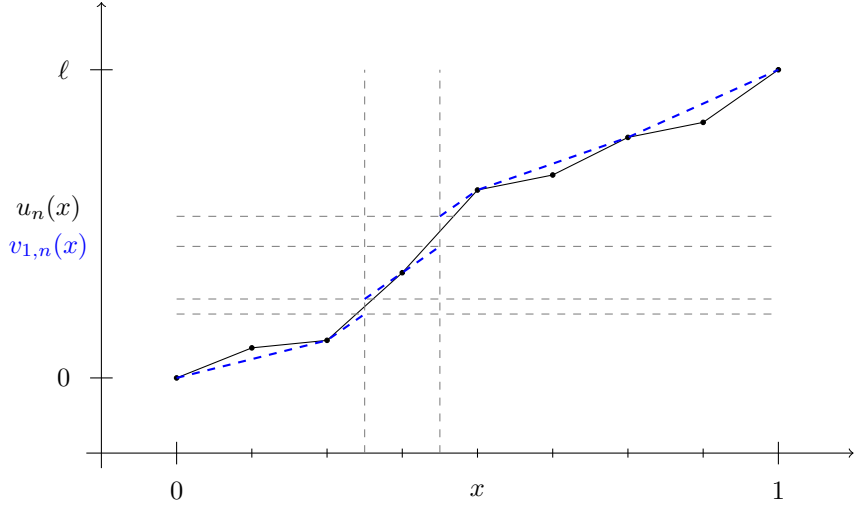

The definition of $v_{2,n}$ is performed using the obvious modifications: on the intervals $((2k-1)\lambda_n, (2k+1) \lambda_n)$ one applies the same procedure to $\widetilde v_{2,n}$ used to construct $v_{1,n}$ from $\widetilde v_{1,n}$.
\vspace*{2mm} \\

%\textcolor{blue}{
%The definition of $v_{2,n}$ is performed using the obvious modifications. One modifies the piecewise affine interpolation $\widetilde v_{2,n}$ defined before by imposing that
%\begin{align*}
%\overline{v}_{1,n}((2k-1)\lambda_n)& :=\widetilde v_{1,n}((2k-1)\lambda_n),\\
%\overline{v}_{1,n}(2k\lambda_n)&:=u_n(2k\lambda_n), \\ %$, and
%\overline{v}_{1,n}((2k+1)\lambda_n)&:=\widetilde v_{1,n}((2k+1)\lambda_n)\, ,
%\end{align*} 
%extending it on $(2k-1)\lambda_n, (2k+1)\lambda_n$ in such a way that the slope equals $\gamma^c$ almost everywhere and possibly performing the final traslation (similarly to $v_{1,n}$) that ensures that the jumps are non-negative.
%}
%\vspace*{2mm} \\
\textbf{Step 2. Compactness estimates}
\vspace*{2mm}\\
Now we are in a position to prove the statements of part~(a) of Theorem~\ref{Compactness1}. By \eqref{utilize} and the definition of $\widetilde v_{1,n}$ above, the following equality holds for any even $0 \leq i \leq n-2$:
\begin{align}
\nonumber
\MoveEqLeft{ \frac{1}{2} J_1\left(\frac{u_n^{i+1}-u_n^{i}}{\lambda_n}\right)
+\frac{1}{2} J_1\left(\frac{u_n^{i+2}-u_n^{i+1}}{\lambda_n}\right)
+J_2\left(\frac{u_n^{i+2}-u_n^{i}}{2\lambda_n}\right) }
\nonumber
\\
%=&
%J_0\left(\frac{u_n^{i+2}-u_n^{i}}{2\lambda_n}\right)
%+ R\left(\frac{u_n^{i+1}-u_n^{i}}{\lambda_n},
%\frac{u_n^{i+2}-u_n^{i+1}}{\lambda_n}\right)
%\\
%\label{ChainVSEvenOddInteraction}
=&
\frac{1}{2\lambda_n}\int_{i\lambda_n}^{(i+2)\lambda_n} J_0({\widetilde v}'_{1,n}) ~dx
+ R\left(\frac{u_n^{i+1}-u_n^{i}}{\lambda_n},
\frac{u_n^{i+2}-u_n^{i+1}}{\lambda_n}\right).
\end{align}
An analogous equality holds for $i$ odd and $\widetilde v_{2,n}$.
 Taking the sum with respect to $i$ ranging from $0$ to $n-2$ and multiplying by $\lambda_n$, we therefore obtain
\begin{align}
\nonumber
\MoveEqLeft{ \sum_{i=0}^{n-1}\lambda_n
J_1\left(\frac{u_n^{i+1}-u_n^{i}}{\lambda_n}\right)
+\sum_{i=0}^{n-2}\lambda_n
J_2\left(\frac{u_n^{i+2}-u_n^{i}}{2\lambda_n}\right) }
\\
\label{Inequ0}
=&
\frac{1}{2}\int_0^1 J_0({\widetilde v}'_{1,n}) ~dx
+\frac{1}{2}\int_0^1 J_0({\widetilde v}'_{2,n}) ~dx
+\sum_{i=0}^{n-2}\lambda_n
R\left(\frac{u_n^{i+2}-u_n^{i+1}}{\lambda_n},
\frac{u_n^{i+1}-u_n^{i}}{\lambda_n}\right)
\\&
\nonumber
+\frac{\lambda_n}{2} \left[J_1\left(\frac{u_n^1-u_n^0}{\lambda_n}\right)
+J_1\left(\frac{u_n^n-u_n^{n-1}}{\lambda_n}\right)
-J_0\left(\frac{u_n^1-u_n^0}{\lambda_n}\right)
-J_0\left(\frac{u_n^n-u_n^{n-1}}{\lambda_n}\right)\right],
\end{align}
where the last line contains the corrections for the segments at the end of the
chain. %Note that $\widetilde v'_{2,n} = \frac{u_n^1-u_n^0}{\lambda_n}$ on $(0,\lambda_n)$ and similarly at the other boundary. 
Moreover as the slope of $u_n$ at
the boundary is prescribed (see assumption $[B1]$), the terms in the last line can be estimated by
$C\lambda_n$. This yields
\begin{align}
\nonumber
\MoveEqLeft{ \sum_{i=0}^{n-1}\lambda_n
J_1\left(\frac{u_n^{i+1}-u_n^{i}}{\lambda_n}\right)
+\sum_{i=0}^{n-2}\lambda_n
J_2\left(\frac{u_n^{i+2}-u_n^{i}}{2\lambda_n}\right) }
\\
\label{Inequ1}
\geq&
\frac{1}{2}\int_0^1 J_0({\widetilde v}'_{1,n}) ~dx
+\frac{1}{2}\int_0^1 J_0({\widetilde v}'_{2,n}) ~dx
+\sum_{i=0}^{n-2}\lambda_n
R\left(\frac{u_n^{i+2}-u_n^{i+1}}{\lambda_n},
\frac{u_n^{i+1}-u_n^{i}}{\lambda_n}\right)
-C\lambda_n
\,.
\end{align}
Notice that for $1 \leq i \leq n-1$ even we have that $u_n^i=v_{1,n}(i\lambda_n)$ and the function $v_{1,n}$ is Lipschitz with slope less or equal than $\gamma$ in $((i-\frac{1}{2})\lambda_n,(i + \frac{1}{2})\lambda_n)$. Therefore by Lipschitz continuity of $\Phi$ (see assumption [$\Phi$1]), we infer
\begin{align}
\frac{1}{\lambda_n}\int_{(i - \frac{1}{2})\lambda_n}^{(i + \frac{1}{2})\lambda_n}&
\Phi(x,v_{1,n}(x)) ~dx  -\Phi(i\lambda_n,u_n^i)\leq \frac{1}{\lambda_n} \int_{(i - \frac{1}{2})\lambda_n}^{(i + \frac{1}{2})\lambda_n}
\left|\Phi(x,v_{1,n}(x)) - \Phi(i\lambda_n,u_n^i)\right|\, dx \nonumber \\
& \leq C\lambda_n+ \frac{C}{\lambda_n}\int_{(i - \frac{1}{2})\lambda_n}^{(i + \frac{1}{2})\lambda_n} |v_{1,n}(x) - u_n^i| ~dx  \leq C\lambda_n\,. 
\label{ChainVSEvenOddForce1}
\end{align}

For $1 \leq i\leq n-1$ odd we can distinguish two cases. If $\frac{u_n^{i+1}-u_n^{i-1}}{2\lambda_n}>\gamma$ we have, by the construction in the previous step that $|u_n^i-v_{1,n}(i\lambda_n)|\leq C \lambda_n$. On the other hand if $\frac{u_n^{i+1}-u_n^{i-1}}{2\lambda_n} \leq \gamma$ we have $v_{1,n}(i\lambda_n) = \widetilde{v}_{1,n}(i\lambda_n) = \frac{u_n^{i+1}+u_n^{i-1}}{2}$. In this case as well it is easy to check that $|u_n^i-v_{1,n}(i\lambda_n)|\leq C \lambda_n$. Therefore, thanks to assumption [$\Phi1$], we infer
\begin{align}
\nonumber & \frac{1}{\lambda_n}\int_{(i-\frac{1}{2})\lambda_n}^{(i+\frac{1}{2})\lambda_n}
\Phi(x,v_{1,n}(x)) ~dx-\Phi(i\lambda_n,u_n^i)\\
& \leq \frac{1}{\lambda_n}\int_{(i-\frac{1}{2})\lambda_n}^{(i+\frac{1}{2})\lambda_n}
\left|\Phi(x,v_{1,n}(x)) -\Phi(i\lambda_n,v_{1,n}(i\lambda_n))\right| ~dx \nonumber
+ \left|
\Phi(i\lambda_n,v_{1,n}(i\lambda_n) -\Phi(i\lambda_n,u_n^i)\right|\\
& \leq C\lambda_n\,.
\label{ChainVSEvenOddForce2}
\end{align}
%On the other hand if $\frac{u_n^{i+1}-u_n^{i-1}}{2\lambda_n}\leq \gamma^c$ we have $v_{1,n}(i\lambda_n) = \frac{u_n^{i+1}+u_n^{i-1}}{2}$. Hence
%\begin{align}
%\frac{1}{\lambda_n}&\left|\int_{(i-\frac{1}{2})\lambda_n}^{(i+\frac{1}{2})\lambda_n}
%\Phi(x,v_{1,n}(x)) ~dx-\Phi(i\lambda_n,u_n^i)\right|\leq \left|\frac{1}{\lambda_n}\int_{(i-\frac{1}{2})\lambda_n}^{(i+\frac{1}{2})\lambda_n}
%\Phi(x,v_{1,n}(x)) ~dx-\Phi(i\lambda_n,v_{1,n}(i\lambda_n))\right| \nonumber\\
%&+ \left| \Phi(i\lambda_n,v_{1,n}(i\lambda_n))-\Phi(i\lambda_n,u_n^i)\right| \leq 
%\lambda_n(C + \gamma^c)  +C\lambda_n
%\left|\frac{u_n^{i+1}-u_n^i}{\lambda_n}-\frac{u_n^{i}-u_n^{i-1}}{\lambda_n}\right| \nonumber \\
%&\leq \lambda_n(C + \gamma^c) + C\lambda_n
%\left|\frac{u_n^{i+1}-u_n^i}{\lambda_n}-\frac{u_n^{i}-u_n^{i-1}}{\lambda_n}\right|\,.
%\label{ChainVSEvenOddForce3}
%\end{align}
Analogous estimates hold
for $v_{2,n}$. 

Then, multiplying inequalities  \eqref{ChainVSEvenOddForce1} and \eqref{ChainVSEvenOddForce2}  by $\lambda_n/2$ and summing from $1$ to $n-1$, we obtain
\begin{align}\label{forceine}
\nonumber \sum_{i=0}^n &\lambda_n \Phi(i\lambda_n,u_n^i) \\
 \geq &\frac{1}{2}\int_0^1 \Phi(x,v_{1,n})\, dx + \frac{1}{2}\int_0^1 \Phi(x,v_{2,n})\, dx +\lambda_n\Phi(0,0)+\lambda_n\Phi(1,\ell) \nonumber \\
& 
-\frac{1}{2}\int_0^{\frac{\lambda_n}{2}} \Phi(x,v_{1,n}) + \Phi(x,v_{2,n}) ~dx
-\frac{1}{2}\int_{1-\frac{\lambda_n}{2}}^1 \Phi(x,v_{1,n}) + \Phi(x,v_{2,n}) ~dx - C\lambda_n \nonumber\\
\geq &\frac{1}{2}\int_0^1 \Phi(x,v_{1,n})\, dx + \frac{1}{2}\int_0^1 \Phi(x,v_{2,n})\, dx - C\lambda_n\,,
\end{align}
where in the last inequality we used assumption $[\Phi1]$. 
Putting together estimates \eqref{forceine} and \eqref{Inequ1}, we infer
\begin{align*}
\MoveEqLeft{ \sum_{i=0}^{n-1}\lambda_n
J_1\left(\frac{u_n^{i+1}-u_n^{i}}{\lambda_n}\right)
+\sum_{i=0}^{n-2}\lambda_n
J_2\left(\frac{u_n^{i+2}-u_n^{i}}{2\lambda_n}\right)
+\sum_{i=0}^n \lambda_n \Phi(i\lambda_n,u_n^i) }
\\
\geq&
\frac{1}{2}\int_0^1 J_0^{**}(v'_{1,n}) +\Phi(x,v_{1,n}) ~dx
+\frac{1}{2}\int_0^1 J_0^{**}(v'_{2,n}) +\Phi(x,v_{2,n}) ~dx
\\&
+\frac{1}{2}\int_0^1 (J_0-J_0^{**})(\widetilde v'_{1,n}) ~dx
+\frac{1}{2}\int_0^1 (J_0-J_0^{**})(\widetilde v'_{2,n}) ~dx
\\&
+\sum_{i=0}^{n-2}\lambda_n
R\left(\frac{u_n^{i+2}-u_n^{i+1}}{\lambda_n},
\frac{u_n^{i+1}-u_n^{i}}{\lambda_n}\right)
-C\lambda_n\, ,
\end{align*}
where we have used $J_0^{**}(v'_{i,n})=J_0^{**}(\widetilde v'_{i,n})$,
as ${\widetilde v'}_{i,n}$ only differs from $v'_{i,n}$ on segments with slope of at
least $\gamma$ and $J_0^{**}(z)$ is constant for all $z\geq \gamma$.

Note that the first two integrals on the right-hand side are precisely
$\frac{1}{2}H(v_{1,n})+\frac{1}{2}H(v_{2,n})$. Subtracting this term
from both sides of the equation and dividing by $\lambda_n$, we obtain that
\begin{align}
\nonumber
H_{1,n}^\ell(u_n) & = \frac{H_n^\ell(u_n) - \inf H(u)}{\lambda_n}  \\
\nonumber & \geq
\frac{H_n^\ell(u_n)
-\frac{1}{2}H(v_{1,n})-\frac{1}{2}H(v_{2,n})}{\lambda_n} 
\\
\nonumber
& \geq
\frac{1}{2\lambda_n}\int_0^1
(J_0-J_0^{**})(\widetilde v'_{1,n}) ~dx +\frac{1}{2\lambda_n}\int_0^1
(J_0-J_0^{**})(\widetilde v'_{2,n}) ~dx
\\&
+
\sum_{i=0}^{n-2}
R\left(\frac{u_n^{i+2}-u_n^{i+1}}{\lambda_n},
\frac{u_n^{i+1}-u_n^{i}}{\lambda_n}\right)
-C \label{estfirst}\, .
\end{align}
Notice that the constant $C$ in \eqref{estfirst} does not depend on the sequence $(u_n)_n \subset \mathcal{A}^\ell_n(0,1)$.

We prove \eqref{expl} by contradiction. 
For a given $0<\varepsilon < \frac{\gamma^c - \gamma}{2}$ define 
\begin{equation*}
I_n := \left\{i: \frac{u^{i+1}_n - u^i_n}{\lambda_n} \geq \gamma + \varepsilon\right\}\, ,
\end{equation*}
\begin{equation}\label{def-Jn}
J_n := \left\{i: \frac{u_n^{i+2} - u_n^i}{2\lambda_n} \geq \gamma + \varepsilon/2\right\}\,.
\end{equation}
Suppose by contradiction that there exists $\varepsilon>0$ such that 
\begin{equation}\label{biggereps}
\limsup_{n\rightarrow + \infty} \#I_n  = + \infty\, .
\end{equation}
First we prove that \eqref{biggereps} implies that
\begin{equation}\label{aver}
\limsup_{n\rightarrow + \infty} \# J_n  = + \infty\, .
\end{equation}
Indeed, using \eqref{estfirst}  and the equiboundedness of the rescaled energy,  we obtain 
\begin{equation}\label{aaaa}
\sum_{i=0}^{n-2}R\left(\frac{u_n^{i+2} - u_n^{i+1}}{\lambda_n}, \frac{u_n^{i+1} - u_n^{i}}{\lambda_n}\right) \leq C\, .
\end{equation}
Introduce the notation
\begin{align*}
& \mathfrak{I}_n := \left\{i :\left|\frac{u_n^{i+1} - u_n^{i}}{\lambda_n} - \frac{u_n^{i+2} - u_n^{i}}{2\lambda_n}\right| \geq \frac{\varepsilon}{2}\right\}\ \cap  J_n^c\, ,
\end{align*}
where we denote by $J_n^c$ the set of indices that do not belong to $J_n$.
By the definition of $J_0$ we have that $R(z_1,z_2) \geq 0$ for every $z_1,z_2$. In addition to that, whenever $\frac{1}{2}(z_1 + z_2) \in (0,\gamma^c)$, we also have that $R(z_1,z_2) = 0$ if and only if $z_1 = z_2 = \frac{1}{2}(z_1 + z_2)$ by [H2].
Therefore, as $\varepsilon$ is chosen such that $\varepsilon < \frac{\gamma^c - \gamma}{2}$ and thanks to hypotheses [H1] and [H5], there exists a constant $C(\varepsilon)>0$ not depending on $n$ such that for every $i\in \mathfrak{I}_n$
\begin{displaymath}
R\left(\frac{u_n^{i+2} - u_n^{i+1}}{\lambda_n}, \frac{u_n^{i+1} - u_n^{i}}{\lambda_n}\right) \geq C(\varepsilon)\,.
\end{displaymath}
Hence from \eqref{aaaa} we infer
\begin{equation*}
C \geq \sum_{i=0}^{n-2}R\left(\frac{u_n^{i+2} - u_n^{i+1}}{\lambda_n}, \frac{u_n^{i+1} - u_n^{i}}{\lambda_n}\right) \geq \sum_{i\in \mathfrak{I}_n}R\left(\frac{u_n^{i+2} - u_n^{i+1}}{\lambda_n}, \frac{u_n^{i+1} - u_n^{i}}{\lambda_n}\right) \geq  (\#\mathfrak{I}_n) C(\varepsilon)\,.
\end{equation*}
We deduce that $\# \mathfrak{I}_n < C$, where the constant $C$ is not depending on $n$. 
By the definition of $I_n$ and $J_n$ we have also that
\begin{align*}
 I_n \cap J_n^c& = \left\{i\in I_n \cap J_n^c:\left|\frac{u_n^{i+1} - u_n^{i}}{\lambda_n} - \frac{u_n^{i+2} - u_n^{i}}{2\lambda_n}\right| \geq \frac{\varepsilon}{2}\right\} = I_n \cap \mathfrak{I}_n\,.
\end{align*}
As $\# \mathfrak{I}_n < C$, we conclude that $\#(I_n \cap J_n^c) <C$.
Then \eqref{aver} follows as a consequence of \eqref{biggereps}.

%As $\gamma + \varepsilon < \gamma^c$, thanks to hypothesis [H2], we get that there exists $n_0\in \N$ such that for $n>n_0$ 
%\begin{equation}\label{almostsameslope}
%\left|\frac{u_n^{i+2} - u_n^{i+1}}{\lambda_n} - \frac{u_n^{i+2} - u_n^{i}}{2\lambda_n}\right| \leq \varepsilon/2 \quad \mbox{and} \quad \left|\frac{u_n^{i+1} - u_n^{i}}{\lambda_n} - \frac{u_n^{i+2} - u_n^{i}}{2\lambda_n}\right| \leq \varepsilon/2
%\end{equation}
%for all but finitely many $i\in I_n \setminus J_n$ (we are considering the points outside $J_n$ because there hypothesis [H2] applies). Therefore $\frac{\widetilde{v}_{1,n}(2i\lambda_n) - \widetilde{v}_{1,n}(2(i-1)\lambda_n)}{2\lambda_n} > \gamma + \varepsilon/2$ for all but finitely many $i\in I_n$. Then \eqref{aver} follows as a consequence of \eqref{biggereps}. 

Finally, using the inequality \eqref{estfirst} we obtain
\begin{displaymath}
\sum_{i\in J_n} \left(J_0\left(\frac{u_n^{i+2} - u_n^{i}}{2\lambda_n}\right) - J_0(\gamma)\right) \leq C\, .
\end{displaymath}
By the definition of $J_n$ in \eqref{def-Jn} we have
\begin{displaymath}
\sum_{i\in J_n} \left(J_0\left(\frac{u_n^{i+2} - u_n^{i}}{2\lambda_n}\right) - J_0(\gamma)\right) \geq \sum_{i\in J_n} \left(\min_{z \geq \gamma+\varepsilon/2 }J_0\left(z\right) - J_0(\gamma)\right).
\end{displaymath}
Hence assumption [H4] and \eqref{aver} yield a contradiction.

In an analogous way it is possible to obtain \eqref{eq17-?}.

\end{proof}

\subsection{Proof of part (b) of Theorem~\ref{Compactness1}}\label{b}
\begin{proof}

In order to obtain \eqref{eq17-1} we estimate in the following way: 
\begin{align*}
& \sum_{i\in \mathcal{I}_n} \left(\frac{u_n^{i+1} - u_n^i}{\lambda_n} - \gamma \right)_+^2 \\
&=
\sum_{i\in \mathcal{I}_n} \min\left\{\left(\frac{u_n^{i+1} - u_n^i}{\lambda_n} - \gamma \right)_+^2, (\gamma^c - \gamma)^2\right\}\\
&= \sum_{i\in \mathcal{I}_n} \min\left\{\left(\frac{u_n^{i+1} - u_n^{i-1}}{2\lambda_n} - \gamma + \frac{u_n^{i+1} - u_n^{i}}{2\lambda_n} - \frac{u_n^{i} - u_n^{i-1} }{2\lambda_n} \right)_+^2, (\gamma^c - \gamma)^2\right\}\\
& \leq \sum_{i\in \mathcal{I}_n}  \min\left\{2\left(\frac{u_n^{i+1} - u_n^{i-1}}{2\lambda_n} - \gamma\right)_+^2 + 2\left(\frac{u_n^{i+1} - u_n^{i}}{2\lambda_n} - \frac{u_n^{i} - u_n^{i-1} }{2\lambda_n} \right)_+^2 , (\gamma^c - \gamma)^2\right\}\\
&\leq \sum_{i\in \mathcal{I}_n} \min\left\{2\left(\frac{u_n^{i+1} - u_n^{i-1}}{2\lambda_n} - \gamma\right)_+^2, (\gamma^c - \gamma)^2\right\} \\ &\quad + \sum_{i\in \mathcal{I}_n}\min\left\{2\left(\frac{u_n^{i+1} - u_n^{i}}{2\lambda_n} - \frac{u_n^{i} - u_n^{i-1} }{2\lambda_n} \right)_+^2 , (\gamma^c - \gamma)^2\right\}\\
&\leq 2\sum_{i\in \mathcal{I}_n} \min\left\{\left(\frac{u_n^{i+1} - u_n^{i-1}}{2\lambda_n} - \gamma\right)_+^2, (\gamma^c - \gamma)^2\right\} + \sum_{i\in \mathcal{I}_n}\left|\frac{u_n^{i+1} - u_n^{i}}{\lambda_n} - \frac{u_n^{i} - u_n^{i-1} }{\lambda_n} \right|^2\,.
\end{align*}
Using \eqref{expl} in part (a) of Theorem~\ref{Compactness1}, we deduce that for all $i\in \mathcal{I}_n$ but a finite number of indices (independent of $n$) one has that
\begin{equation*}
\frac{u_n^{i} - u_n^{i-1}}{\lambda_n} < \gamma^c\, ,
\end{equation*} 
where $c$ is the constant defined in assumption [H0].
Therefore, combining this fact with Lemma~\ref{Jonecomp} and inequality \eqref{estfirst} we infer that
\begin{displaymath}
\sum_{i\in \mathcal{I}_n}\left|\frac{u_n^{i+1} - u_n^{i}}{\lambda_n} - \frac{u_n^{i} - u_n^{i-1} }{\lambda_n} \right|^2 < C\,.
\end{displaymath}
Notice that
\begin{displaymath}
\mathfrak{I}_n := \left\{i: \frac{u_n^{i+1} - u_n^{i-1}}{2\lambda_n} > \gamma^c\right\} \subset \left(\left\{i: \frac{u_n^{i+1} - u_n^{i}}{\lambda_n} > \gamma^c\right\} \cup \left\{i: \frac{u_n^{i} - u_n^{i-1}}{\lambda_n} > \gamma^c\right\}\right)\, ;
\end{displaymath}
hence, thanks to \eqref{expl} in Theorem~\ref{Compactness1}, its cardinality is finite and independent of $n$. Therefore
\begin{equation}\label{quasi}
\sum_{i\in \mathcal{I}_n} \left(\frac{u_n^{i+1} - u_n^i}{\lambda_n} - \gamma \right)_+^2 \leq 2\sum_{i\in \mathfrak{I}_n^c} \left(\frac{u_n^{i+1} - u_n^{i-1}}{2\lambda_n} - \gamma\right)_+^2 + C\, .
\end{equation}
Finally, using the strict convexity of $J_0$ in $(0,\gamma^c + c)$ (see hypothesis [H0]), we obtain from \eqref{quasi} that
\begin{equation*}
\sum_{i\in \mathcal{I}_n} \left(\frac{u_n^{i+1} - u_n^i}{\lambda_n} - \gamma \right)_+^2 \leq  C \sum_{i\in \mathfrak{I}_n^c} \left(J_0\left(\frac{u_n^{i+1} - u_n^{i-1}}{2\lambda_n}\right) - J_0^{\ast\ast}\left(\frac{u_n^{i+1} - u_n^{i-1}}{2\lambda_n}\right)\right) + C\\
\leq C\, ,
\end{equation*}
where in the last inequality we used \eqref{estfirst} and the equiboundedness of the rescaled energy.

\end{proof}

\subsection{Proofs of Theorem~\ref{Compactness1}, part (c), and Propositions~\ref{Compactness3} and \ref{LimitationsCondition}}\label{c}

\begin{proof}[\textbf{Proof of part (c) of Theorem~\ref{Compactness1}}]

Assume that \eqref{liminfslopebigger}
was not true.  
Notice first that, as $J_j(z)=+\infty$ for $z\leq 0$ (see [H5]), we may restrict ourselves to sequences of
configurations $(u_n)_n$ of the discrete chains with $u_n\rightarrow u$ with $u'_n>0$.
Set $0<S<\widetilde S$ and define
\begin{displaymath}
h(x) := \int_{\lambda_n}^x \chi_{\{u'_n < S\}} ~dt \quad \mbox{and} \quad w(x) :=\int_{\lambda_n}^x \chi_{\{u'_n > \widetilde S\}}~dt\, .
\end{displaymath}
for $x\in (\lambda_n,1-\lambda_n)$ and extended constantly (and in a continuous way) to $(0,1)$.
Then setting $K>0$, define 
\begin{displaymath}
\widetilde u_n(x) := u_n(x) + Kh(x) - \frac{K h(1)w(x)}{w(1)}\, .
\end{displaymath}
One can easily check that $\widetilde u_n \in \mathcal{A}^\ell_n(0,1)$. Moreover
\begin{displaymath}
\widetilde u_n'(x) = u_n'(x) + K  \chi_{\{u'_n < S\}} - \dfrac{K h(1)\chi_{\{u'_n > \widetilde S\}}}{w(1)}\, .
\end{displaymath}
Notice that $h(1) \rightarrow 0$ as $S\rightarrow 0$ uniformly in $n$. Indeed if $h(1)$ was bounded away from zero as $S\rightarrow 0$ and along a subsequence $n_k \rightarrow 0$, then using a diagonal argument it is easy to see that $\sup_n H_{n}(u_{n}) = +\infty$. 
Moreover for a similar argument it is possible to choose $\widetilde S$ small enough such that $w(1) > C > 0$ uniformly in $n$.

Hence choosing $S$ small enough such that $Kh(1)/w(1) \leq \widetilde S /2$ we ensure that $\widetilde u_n'(x) > 0$ and therefore $\widetilde u_n \in \mathcal{A}^\ell_n(0,1)$.
Moreover
\begin{displaymath}
|\widetilde u_n - u_n| = K\left|h(x) - h(1)\dfrac{w(x)}{w(1)}\right|\leq Kh(1).
\end{displaymath}
This implies, thanks to assumption $[\Phi1]$ that
\begin{gather*}
\left|\sum_{i=0}^n \lambda_n \Phi(i \lambda_n,u_n^i)-\sum_{i=0}^n
\lambda_n \Phi(i \lambda_n,{\widetilde u}_n^i)\right|\leq Ch(1)
\,.
\end{gather*}
Therefore
\begin{align}
H_n(u_n)-H_n(\widetilde u_n) \geq &\sum_{i=0}^{n-1} \left[J_1 \left(\frac{u_n^{i+1} - u_n^i}{\lambda_n} \right)- J_1 \left(\frac{\widetilde u_n^{i+1} - \widetilde u_n^i}{\lambda_n} \right)\right] \nonumber \\
&+ \sum_{i=0}^{n-2} \left[J_2 \left(\frac{u_n^{i+2} - u_n^i}{\lambda_n}\right) - J_2 \left(\frac{\widetilde u_n^{i+2} - \widetilde u_n^i}{\lambda_n} \right)\right] - C h(1)\, . \label{spliit}
\end{align}
Among the contribution of the first neighborhood interaction in the previous estimate, we treat the intervals with $u_n' \geq \widetilde S$ and the intervals with $u_n' < S$ differently (if $ S<u_n' \leq \widetilde S$, then $\widetilde u_n' = u_n'$).
First consider the intervals $(i\lambda_n,(i+1)\lambda_n)$ where $u'_n \geq \widetilde S$. In these intervals we have $|\widetilde u_n' - u_n'| \leq Kh(1)/w(1) \leq \widetilde S / 2$ for $n$ big enough. Hence thanks to hypothesis [H1] we have
\begin{gather*}
J_1\left(\frac{u_n^{i+1}-u_n^{i}}{\lambda_n}\right)
- J_1\left(\frac{{\widetilde u}_n^{i+1}-{\widetilde u}_n^{i}}{\lambda_n}\right) \geq -Ch(1)\,.
\end{gather*}
On the other hand, if $u'_n <  S$ on $(i\lambda_n,(i+1)\lambda_n)$, then $K < \widetilde u_n' < K +S$ on these segments.  
Therefore, thanks to the convexity of $J_1$ in assumption [H0] (we choose $K + S$ small enough to make $J_1$ monotone in $(0,K+S)$), we obtain 
\begin{gather*}
J_1\left(\frac{u_n^{i+1}-u_n^{i}}{\lambda_n}\right)
- J_1\left(\frac{{\widetilde u}_n^{i+1}-{\widetilde u}_n^{i}}{\lambda_n}\right)
\geq
J_1(S)
- J_1(K)\,
\end{gather*}
on these segments. 

Regarding the
next-to-nearest neighbour potentials, we split the sum \eqref{spliit} into the sum of the intervals $(i\lambda_n, (i+1)\lambda_n)$ such that $u_n'\geq \widetilde S$ either on $(i\lambda_n,(i+1)\lambda_n)$ or
$((i+1)\lambda_n,(i+2)\lambda_n)$,
where we have (using again hypothesis [H1])
\begin{gather*}
J_2\left(\frac{u_n^{i+2}-u_n^{i}}{2\lambda_n}\right) - J_2\left(\frac{\widetilde u_n^{i+2}- \widetilde u_n^{i}}{2\lambda_n}\right)
\geq -Ch(1)\, ,
\end{gather*}
and in the intervals such that
$u'_n<S$ either on $(i\lambda_n,(i+1)\lambda_n)$ or $((i+1)\lambda_n,(i+2)\lambda_n)$, where
\begin{gather*}
J_2\left(\frac{u_n^{i+2}-u_n^{i}}{2\lambda_n}\right)
-J_2\left(\frac{{\widetilde u}_n^{i+2}-{\widetilde u}_n^{i}}{2\lambda_n}\right)
\geq -C
\end{gather*}
as $J_2$ is bounded from below. In the remaining cases we have $u_n' = \widetilde u_n'$ in $(i\lambda_n, (i+2) \lambda_n)$.
Hence from inequality \eqref{spliit} we obtain
\begin{equation*}
{H_n(u_n)-H_n(\widetilde u_n)}
\geq 
h(1)\left(J_1(S)- J_1(K) - C\right)  - Ch(1)\, ,
\end{equation*}
where $C$ does neither depend on $n$ nor on $S$.

Notice in addition that
\begin{displaymath}
h(1) = \lambda_n \# \left\{i:\frac{u_n^{i+1}-u_n^i}{\lambda_n}<S\right\}\, .
\end{displaymath}
Hence
\begin{equation}\label{aa}
\limsup_n \frac{H_{n}(u_{n})-H_{n}(\widetilde u_{n})}{\lambda_{n}} \geq 
\Big[J_1(S)-J_1(K) - C \Big] \limsup_n\# \left\{i:\frac{u_{n}^{i+1}-u_{n}^i}{\lambda_{n}}<S\right\}\,.
\end{equation}
As we assume that \eqref{liminfslopebigger} does not hold, we have that
\begin{gather*} 
\liminf_{n\rightarrow +\infty} \min_{i\in \{0,\ldots,n-1\}}
\frac{u_n^{i+1}-u_n^i}{\lambda_n}=0\,.
\end{gather*}
This implies that for every $S>0$ it holds 
\begin{equation}\label{op}
\limsup_n \# \left\{i:\frac{u_n^{i+1}-u_n^i}{\lambda_n}<S\right\}\geq 1\,.
\end{equation}
By extracting a subsequence from $(u_n)_n$ (denoted again by $u_n$), we can suppose that the $\limsup$ in \eqref{op} and 
\begin{eqnarray*}
\limsup_n \frac{H_{n}(u_{n})-H_{n}(\widetilde u_{n})}{\lambda_{n}}
\end{eqnarray*}
are both realized. Then
\begin{eqnarray}
\liminf_n H_{1,n}^\ell(u_n) & \geq & \liminf_n \frac{H_n(u_n)-H_n(\widetilde
u_n)}{\lambda_n}+\liminf_n H_{1,n}^\ell(\widetilde u_n) \nonumber \\
&\geq &\liminf_n \frac{H_n(u_n)-H_n(\widetilde
u_n)}{\lambda_n} -M \nonumber\\
& =& \lim_n \frac{H_n(u_n)-H_n(\widetilde
u_n)}{\lambda_n} -M\, , \label{op2}
\end{eqnarray}
where we used \eqref{estfirst} in the proof of the part a) of Theorem~\ref{Compactness1} to estimate $H_{1,n}^\ell(\widetilde u_n)$ from below by $-M$ (notice that in \eqref{estfirst} all the terms on the right side of the inequality are positive). We remark that $M$ is independent of the sequence $\widetilde u_n$ and more precisely of $S$. 
Finally choosing $S$ small enough in \eqref{aa} and using hypothesis [H5] and the equiboundness of $H_{1,n}^\ell(u_n)$ in combination with \eqref{op} and \eqref{op2}, we reach a contradiction to the equiboundedness of $H_{1,n}^\ell$.

\end{proof}

\begin{proof}[\textbf{Proof of Proposition~\ref{Compactness3}}]

%By Lemma~\ref{convexest}, $R(y,z)\geq c|y-z|^2$ for $\frac{y+z}{2}\leq \gamma^c$ and
%$y,z>0$. Thus, using Young's inequality the
%term in the last line we obtain that for any $\eta >0$
%\begin{displaymath}
%C\sum_{i\in I_n} \lambda_n
%\left|\frac{u_n^{i+1}-u_n^i}{\lambda_n}
%-\frac{u_n^{i}-u_n^{i-1}}{\lambda_n}\right| \leq \frac{C}{\eta}\sum_{i=1}^{n} \lambda_n^2 + C\eta\sum_{i=0}^{n-2} 
%R \left(\frac{u_n^{i+1} - u_n^i}{\lambda_n}, \frac{u_n^{i} - u_n^{i-1}}{\lambda_n} \right)\,.
%\end{displaymath}
%Therefore choosing $\eta$ small enough we get
%\begin{align}
%& H_{1,n}^\ell(u_n)
%\geq
%\frac{1}{2\lambda_n}\int_0^1
%(J_0-J_0^{**})(\widetilde v'_{1,n}) ~dx +\frac{1}{2\lambda_n}\int_0^1
%(J_0-J_0^{**})(\widetilde v'_{2,n}) ~dx \nonumber \\
%& \qquad\qquad +C
%\sum_{i=0}^{n-2}
%R\left(\frac{u_n^{i+2}-u_n^{i+1}}{\lambda_n},
%\frac{u_n^{i+1}-u_n^{i}}{\lambda_n}\right) -C - C\lambda_n\, . \label{estfirst} 
%\end{align}
%Let us prove \eqref{eq17-?}. Thanks to \eqref{eq17-2} we have
%\begin{equation*}
%\sum_{i=1}^{n-2}\frac{1}{2} \Big[J_1\left(\frac{u_n^{i+2} - u_n^{i+1}}{\lambda_n}\right) + J_1\left(\frac{u_n^{i+1} - u_n^{i}}{\lambda_n}\right)\Big] + J_2\left(\frac{u_n^{i+2} - u_n^{i}}{2\lambda_n}\right) - J_0\left(\frac{u_n^{i+2} - u_n^{i}}{2\lambda_n}\right) \leq C.
%\end{equation*}

Before we show the convergence properties of $u_n$, we observe that, with the same argument as in the proof of Proposition~\ref{MinimizerExistence}, one obtains that $\|u_n\|_{BV(0,1)} \leq C$ uniformly in $n$. Therefore, there exists a subsequence (again denoted by) $(u_n)_n$ and $u\in BV^\ell([0,1])$ such that $u_n \rightharpoonup u$ weakly* in $BV$. In particular, we have that $u_n \rightarrow u$ strongly in $L^1$.
Notice in addition that, thanks to assumption [H5] and the equiboundedness of the energy, $(u_n^{i+1} - u_n^i)/\lambda_n > 0$ for every $i$ and $n$.

We describe in details the arguments for $n$ even; the modifications for $n$ odd are straightforward. 

Consider
\begin{equation*}
I_n := \left\{i\in \{0,\ldots,n-2\} : \frac{u_n^{i+2} - u_n^i}{2\lambda_n} > \sqrt{n}\right\}
\end{equation*}
and define the following function:
\begin{equation*}
v_n(x) := \left\{
\begin{array}{ll}
u_n(x)  & x\in [i\lambda_n, (i+2)\lambda_n), i\notin I_n\, ,\\
u_n(i\lambda_n) & x\in [i\lambda_n, (i+2)\lambda_n), i\in I_n\, . 
\end{array}
\right.
\end{equation*}
Let us prove that $v_n \rightarrow u$ in $L^1$ as $n\rightarrow +\infty$. Indeed, by construction
\begin{displaymath}
\int_0^1 |u_n - v_n|\, dx = \sum_{i\in I_n} \int_{i\lambda_n}^{(i+2)\lambda_n} |u_n - u_n(i\lambda_n)|\, dx \leq   \sum_{i\in I_n} \lambda_n (u_n((i+2)\lambda_n) - u_n(i\lambda_n)) \leq \lambda_n \ell\, .
\end{displaymath}
Defining 
\begin{equation*}
\mathcal{H}_n(u_n) = \sum_{i=1}^{n-2} \left[J_0\left(\frac{u_n^{i+2} - u_n^{i}}{2\lambda_n}\right) - J^{\star\star}_0\left(\frac{u_n^{i+2} - u_n^{i}}{2\lambda_n}\right)\right]\, ,
\end{equation*}
we have, thanks to \eqref{estfirst}, that $\sup_{n} \mathcal{H}_n(u_n) < C$ and therefore $\sup_n \# I_n < M$ for $M>0$. Hence we can suppose that $S_{v_n}$ converges to a finite set that we denote by $\{x_1,\ldots,x_m\} \subset [0,1]$. As the nature of the following argument is local, we can assume without loss of generality that $S = \{x_0\}$.
Define the following sequence of functions:
\begin{equation*}
w_n(x) = \left\{
\begin{array}{ll}
\displaystyle v_n(0) + \int_0^{x} v_n'(t) \, dt & x \leq x_0 \, ,\\[1em]
\displaystyle v_n(0) + \int_0^{x} v_n'(t) + \sum_{t \in S_{v_{n}}} [v_n(t)] & x>x_0 \, .
\end{array}
\right.
\end{equation*}
First we prove that $w_n \rightarrow u$ almost everywhere in $(0,1)$. Indeed, for every $\varepsilon > 0$ there exists an $n_0$ such that for all $n>n_0$ we have $\mbox{dist}(S_{v_n},x_0) < \varepsilon$. Hence $w_n = v_n$ in $(x_0 - \varepsilon, x_0 + \varepsilon)$. As $\varepsilon$ is arbitrary, it follows that $v_n$ and $w_n$ have the same pointwise limit. 

Notice that $w_n' = v_n'$ by construction; moreover $v_n'(x) = u_n'(x)$ for $x\in (i\lambda_n, (i+2)\lambda_n)$ with $i\notin I_n$ and $v_n'=0$ otherwise.
This allows to deduce that there exist $c_1,c_2> 0$ such that 
\begin{equation}
c_1\int_0^1\left|v_n'\right|^2\, dx - c_2 \leq \mathcal{H}_n(u_n) < C\,. \label{uff}
\end{equation}
Indeed, given $C > 0$, by the construction of $v_n$, we have
\begin{equation*}
\int_0^1\left|v_n'\right|^2 \, dx = \int_{\{v_n'\leq \gamma + C\}} \left|v_n'\right|^2 \, dx + \int_{\{v_n'> \gamma + C\}} \left|v_n'\right|^2\, dx \leq (\gamma + C)^2 + n |\{v_n' > \gamma + C\}|\, .
\end{equation*}
Now, set $I^C_n := \{i: v_n' > \gamma + C \mbox{ in } (i\lambda_n, (i+2)\lambda_n)\}$. Then there exists $\widetilde C>0$ such that (see hypothesis [H4])
\begin{eqnarray*}
\mathcal{H}_n(u_n) &=& \sum_{i=0}^{n-2} J_0\left[\left(\frac{u_n^{i+2} - u_n^i}{2\lambda_n}\right) - J_0^{\star\star}\left(\frac{u_n^{i+2} - u_n^i}{2\lambda_n}\right)\right] \geq \sum_{i\in I_n} \left[J_0\left(\frac{u_n^{i+2} - u_n^i}{2\lambda_n}\right) - J_0\left(\gamma\right)\right]\\
&\geq& \widetilde C \#I^C_n 
= \widetilde Cn |\{v_n' > \gamma + C\}|  \,.
\end{eqnarray*}

Therefore \eqref{uff} holds with the right choice of the constants $c_1$ and $c_2$.

Equation \eqref{uff} implies that $\|w_n'\|_{L^2(0,1)} = \|v_n'\|_{L^2(0,1)} < C$ and, applying Poincar\'e's inequality on the intervals $(0,x_0)$ and $(x_0,1)$, we infer that $w_n$ is uniformly bounded in $W^{1,2}((0,1)\setminus \{x_0\})$. Hence $u\in W^{1,2}((0,1) \setminus \{x_0\})$ and, up to subsequences, $w_n'\rightharpoonup u'$ weakly in $L^2(0,1)$. In addition, the definitions of $v_n, w_n$ yield that
\begin{equation*}
\lim_{n\rightarrow +\infty} [w_n](x_0) = [u](x_0)\,. 
\end{equation*}
Finally, following \cite{BraidesGelli2002}, we define the new potential $\widetilde J_0(z) := J_0(z) - J_0^{\star\star}(z)$ for $z>0$ and the rescaled (and extended) ones:
\begin{equation}
F_n(z) = \left\{
\begin{array}{ll}
\dfrac{\widetilde J_0(z)}{\lambda_n} & 0 < z \leq \sqrt{n}\, ,\\
+\infty & \mbox{otherwise}\, ,
\end{array}
\right.
\end{equation}
and 
\begin{equation}
G_n(z) = \left\{
\begin{array}{ll}
\widetilde J_0\left(\dfrac{z}{\lambda_n}\right)  & z\geq \dfrac{1}{\sqrt{n}} \, ,\\ 
+\infty & \mbox{otherwise}\, .
\end{array}
\right.
\end{equation}
We observe that 
\begin{eqnarray*}
\mathcal{H}_n(u_n) &=& \sum_{i\notin I_n} \widetilde J_0\left(\frac{v_n^{i+2} - v_n^{i}}{2\lambda_n}\right) + \sum_{i\in I_n}\widetilde J_0\left(\frac{[v_n]((i+2)\lambda_n)}{2 \lambda_n}\right)  \\
&=& \int_0^1 F_n(v_n')\, dx + \sum_{t\in S_{v_n}} G_n\left([v_n](t)\right)\, .
\end{eqnarray*}

By classical results of $\Gamma$-convergence, see for example \cite[Proposition~2.2]{BraidesGelli2002}, generalized for an arbitrary number of discontinuities, we obtain 
\begin{displaymath}
C > \liminf_n \mathcal{H}_n(u_n) \geq \int_0^1 F(u')\, dx + \sum_{t\in S_u} G([u](t)) \quad \mbox{ if } u\in SBV^\ell([0,1]) \, ,
\end{displaymath}
where 
\begin{displaymath}
F(z) = \left\{
\begin{array}{ll}
0 & 0 <z \leq \gamma\\
+\infty & \mbox{otherwise}\, 
\end{array}
\right.
\end{displaymath}
and
\begin{displaymath}
G(w) = \left\{
\begin{array}{ll}
J_0(+\infty) - J_0(\gamma) & w \geq \gamma\\
+\infty & w \leq 0\\
0 & \mbox{otherwise}\,.
\end{array}
\right.
\end{displaymath}
Moreover $\liminf_n \mathcal{H}_n(u_n) = +\infty$ if $u\in L^1(0,1) \setminus SBV^\ell([0,1])$.

Hence we deduce that $u\in SBV^\ell([0,1])$, $[u]>0$ and $u' \leq \gamma$ almost everywhere and, thanks to hypothesis [H5], that $\#S_u < + \infty$.

\end{proof}

\begin{proof}[\textbf{Proof of Proposition~\ref{LimitationsCondition}}]

By Proposition~\ref{MinimizerExistence} we know that a minimizer $u$ of $H$ exists. We show that we may construct a minimizer $\widetilde u$ such that $\widetilde u'=\min\{u',\gamma\}$ holds and such that $D^s \widetilde u$ is concentrated on the set of points where $\frac{\partial\Phi}{\partial u}$ changes sign and at the boundary of the interval $(0,1)$.

By assumption, $\frac{\partial\Phi}{\partial u}$ changes sign at at most finitely many points, which we denote by $x_1<x_2<\ldots<x_{M-1}$ for some $M\in \mathbb{N}$. Set $x_0=0$ and $x_M=1$. Then define 
\begin{equation}
\widetilde u(x):=\int_{x_i}^x \min\{u'(y),\gamma\} \,dy+a_i
\end{equation}
for $x\in (x_i,x_{i+1})$, $i=0,\ldots,M-1$, where we set
\begin{align*}
a_i:=u(x_i+)
\end{align*}
if $\Phi(x,w)$ is nondecreasing in $w$ for $x\in (x_i,x_{i+1})$, and
\begin{align*}
a_i:=u(x_i+)+\int_{x_i}^{x_{i+1}} (u'(y)-\gamma)_+ \,dy + |D^s u| (x_i,x_{i+1})
\end{align*}
if $\Phi(x,w)$ is nonincreasing in $w$ for $x\in (x_i,x_{i+1})$. Additionally, we add jumps at the boundary of the interval $(0,1)$ in order to match the boundary conditions. We have $\widetilde u(x)\leq u(x)$ on $(x_i,x_{i+1})$ if $\Phi(x,w)$ is nondecreasing in $w$ and the reverse estimate if $\Phi(x,w)$ is nonincreasing in $w$. Hence
\begin{align*}
H(u)-H(\widetilde u)=\sum_{i=0}^{M-1} \int_{x_i}^{x_{i+1}} \Phi(x,u(x))-\Phi(x,\widetilde u(x)) \,dx \geq 0\, .
\end{align*}
So, $\widetilde u \in BV^\ell([0,1])$ is also a minimizer of $H$.

\end{proof}

\subsection*{Acknowledgements}
 This project was partially supported through grants SCHL 1706/2-1 and SCHL 1706/4-1 of the German Science Foundation. 
 All authors acknowledge the stimulating working conditions at the Erwin Schr\"odinger International Institute for Mathematics and Physics (ESI), where   part of this research was developed during the workshop ``New trends in the variational modeling of failure phenomena''. AS would like to thank the Isaac Newton Institute for Mathematical Sciences for support and hospitality during the programme
``The Mathematical Design of New Materials''
when work on this paper was undertaken. This programme was supported by EPSRC grant number EP/R014604/1.

%\newpage
\bibliographystyle{abbrv}
\bibliography{boundarylayer}

\begin{thebibliography}{10}

\bibitem{AmbrosioFuscoPallara}
L.~Ambrosio, N.~Fusco, and D.~Pallara.
\newblock {\em Functions of bounded variation and free discontinuity problems}.
\newblock Oxford Mathematical Monographs. The Clarendon Press, Oxford
  University Press, New York, 2000.

\bibitem{AnzellottiBaldo}
G.~Anzellotti and S.~Baldo.
\newblock Asymptotic development by {$\Gamma$}-convergence.
\newblock {\em Appl. Math. Optim.}, 27:105--123, 1993.

\bibitem{Braides}
A.~Braides.
\newblock {\em {Gamma-convergence for beginners}}.
\newblock Oxford University Press, 2002.

\bibitem{BraidesCicalese}
A.~Braides and M.~Cicalese.
\newblock {Surface energies in nonconvex discrete systems}.
\newblock {\em Math. Models Meth. Appl. Sci.}, 17:985--1037, 2007.

\bibitem{BraidesDalmasoGarroni}
A.~Braides, G.~Dal~Maso, and A.~Garroni.
\newblock Variational formulation of softening phenomena in {F}racture
  {M}echanics: the one-dimensional case.
\newblock {\em Arch. Rational Mech. Anal.}, 146:23--58, 1999.

\bibitem{BraidesGelli2002}
A.~Braides and M.~S. Gelli.
\newblock Continuum limits of discrete systems without convexity hypotheses.
\newblock {\em Math. Mech. Solids}, 7(1):41--66, 2002.

\bibitem{BraidesGelli}
A.~Braides and M.~S. Gelli.
\newblock The passage from discrete to continuous variational problems: a
  nonlinear homogenization process.
\newblock {\em Nonlinear Homogenization and its Applications to Composites,
  Polycrystals and Smart Materials}, pages 45--63, 2004.

\bibitem{BraidesGelli2017}
A.~Braides and M.~S. Gelli.
\newblock Analytical treatment for the asymptotic analysis of microscopic
  impenetrability constraints for atomistic systems.
\newblock {\em ESAIM Math. Model. Numer. Anal.}, 51(5):1903--1929, 2017.

\bibitem{BraidesLewOrtiz}
A.~Braides, A.~J. Lew, and M.~Ortiz.
\newblock Effective cohesive behavior of layers of interatomic planes.
\newblock {\em Arch. Rational Mech. Anal.}, 180(2):151--182, 2006.

\bibitem{BraidesSolci2014}
A.~Braides and M.~Solci.
\newblock Asymptotic analysis of {L}ennard-{J}ones systems beyond the
  nearest-neighbour setting: a one-dimensional prototypical case.
\newblock {\em Math. Mech. Solids}, 21(8):915--930, 2014.

\bibitem{BraidesTruskinovsky}
A.~Braides and L.~Truskinovsky.
\newblock Asymptotic expansions by {$\Gamma$-}convergence.
\newblock {\em Continuum Mech. Thermodyn.}, 20:21--62, 2008.

\bibitem{ReviewCAW}
C.~S. Casari, M.~Tommasini, R.~R. Tykwinski, and A.~Milani.
\newblock Carbon-atom wires: 1-{D} systems with tunable properties.
\newblock {\em Nanoscale}, 8:4414--4435, 2016.

\bibitem{CharlotteTrusk2002}
M.~Charlotte and L.~Truskinovsky.
\newblock Linear elastic chain with a hyper-pre-stress.
\newblock {\em J. Mech. Phys. Solids}, 50(2):217--251, 2002.

\bibitem{CharlotteTrusk2008}
M.~Charlotte and L.~Truskinovsky.
\newblock Towards multi-scale continuum elasticity theory.
\newblock {\em Contin. Mech. Thermodyn.}, 20(3):133--161, 2008.

\bibitem{Dacorogna}
B.~Dacorogna.
\newblock {\em Direct Methods in the Calculus of Variations}.
\newblock Springer Science+Business Media, LLC, second edition, 2008.

\bibitem{Friedrich2017}
M.~Friedrich.
\newblock A {D}erivation of linearized {G}riffith {E}nergies from {N}onlinear
  {M}odels.
\newblock {\em Arch. Rational Mech. Anal.}, 225:425--467, 2017.

\bibitem{FriedrichSchmidt2014}
M.~Friedrich and B.~Schmidt.
\newblock An atomistic-to-continuum analysis of crystal cleavage in a
  two-dimensional model problem.
\newblock {\em J. Nonlinear Sci.}, 24(1):145--183, 2014.

\bibitem{FriedrichSchmidt2015}
M.~Friedrich and B.~Schmidt.
\newblock An analysis of crystal cleavage in the passage from atomistic models
  to continuum theory.
\newblock {\em Arch. Rational Mech. Anal.}, 217(1):263--308, 2015.

\bibitem{Girifalco}
L.~Girifalco.
\newblock {Molecular properties of C$_{60}$ in the gas and solid phases}.
\newblock {\em J. Phys. Chem.}, 96:858--861, 1992.

\bibitem{Hudson}
T.~Hudson.
\newblock Gamma-expansion for a 1{D} confined {L}ennard-{J}ones model with
  point defect.
\newblock {\em Netw. Heterog. Media}, 8:501--527, 2013.

\bibitem{IosifescuLichtMichaille}
O.~Iosifescu, C.~Licht, and G.~Michaille.
\newblock Variational limit of a one dimensional discrete and statistically
  homogeneous system of material points.
\newblock {\em Asymptot.\ Anal.}, 28:309--329, 2001.

\bibitem{JansenKoenigSchmidtTheil}
S.~Jansen, W.~K{\"o}nig, B.~Schmidt, and F.~Theil.
\newblock Surface energy and boundary layers for a chain of atoms at low
  temperature.
\newblock arXiv:1904.06169.

\bibitem{KitavtsevLuckhausRueland}
G.~Kitavtsev, S.~Luckhaus, and A.~R{\"u}land.
\newblock {Surface energies emerging in a microscopic, two-dimensional two-well
  problem}.
\newblock {\em Proc. Royal Soc. Edinburgh A}, 147(5):1041--1089, 2017.

\bibitem{LauerbachSchaffnerSchlomerkemper}
L.~Lauerbach, M.~Sch\"affner, and A.~Schl\"omerkemper.
\newblock On continuum limits of heterogeneous discrete systems modelling
  cracks in composite materials.
\newblock {\em GAMM-Mitt.}, 40(3):184--206, 2018.

\bibitem{Sylicon}
C.~M. Lieber.
\newblock One-dimensional nanostructures: Chemistry, physics and applications.
\newblock {\em Solid State Comm.}, 107(11):607 -- 616, 1998.

\bibitem{Nairetal}
A.~Nair, S.~Cranford, and M.~Buehler.
\newblock The minimal nanowire: Mechanical properties of carbyne.
\newblock {\em EPL}, 95:16002, 2011, Erratum: \textit{EPL}, 106:39901, 2014.

\bibitem{ScardiaSchloemerkemperZanini}
L.~Scardia, A.~Schl\"omerkemper, and C.~Zanini.
\newblock {Boundary layer energies for nonconvex discrete systems}.
\newblock {\em Math. Models Meth. Appl. Sci.}, 21(4):777--817, 2011.

\bibitem{SSZ12}
L.~Scardia, A.~Schl\"omerkemper, and C.~Zanini.
\newblock Towards uniformly {$\Gamma$}-equivalent theories for nonconvex
  discrete systems.
\newblock {\em Discrete Contin. Dyn. Syst. Ser. B}, 17(2):661--686, 2012.

\bibitem{SchaeffnerSchloemerkemper}
M.~Sch\"affner and A.~Schl\"omerkemper.
\newblock On {L}ennard-{J}ones systems with finite range interactions and their
  asymptotic analysis.
\newblock {\em Netw. Heterog. Media}, 13(1):95--118, 2018.

\bibitem{Wagneretal}
T.~Wagner, J.~Aulbach, J.~Sch{\"a}fer, and R.~Claessen.
\newblock Au-induced atomic wires on stepped ge(hhk) surfaces.
\newblock {\em Phys. Rev. Materials}, 2:123402, 2018.

\bibitem{Warneretal}
J.~H. Warner, Y.~Ito, M.~Zaka, L.~Ge, T.~Akachi, H.~Okimoto, K.~Porfyrakis,
  A.~A.~R. Watt, H.~Shinohara, and G.~A.~D. Briggs.
\newblock Rotating fullerne chains in carbon nanopeapods.
\newblock {\em Nano Letters}, 8:2328--2335, 2008.

\bibitem{Zannoni2001}
C.~Zannoni.
\newblock Molecular design and computer simulations of novel mesophases.
\newblock {\em J. Mater. Chem.}, 11:2637--2646, 2001.

\bibitem{Zhangetal}
G.~Zhang, X.~Fang, Y.~Yao, C.~Wang, Z.~Ding, and K.~Ho.
\newblock Electronic structure and transport of a carbon chain between graphene
  nanoribbon leads.
\newblock {\em J. Phys.: Condens. Matter}, 23:025302, 2011.

\end{thebibliography}

\end{document}